\documentclass[12pt]{article}
\usepackage{amsmath}
\usepackage{tikz} 
\usetikzlibrary{arrows.meta}
\usepackage{tikz}
\usetikzlibrary{arrows.meta,backgrounds,fit}
\usepackage{amsmath}
\usetikzlibrary{arrows.meta, bending, calc}
\usepackage{amssymb,amsfonts,amsmath}
\usepackage{float}
\usepackage{booktabs}
\usepackage{adjustbox}
\usepackage{amsmath,amssymb}
 \usepackage{amsmath}
 \usepackage{tikz} 
 \usetikzlibrary{arrows.meta}
 \usepackage{tikz}
 \usetikzlibrary{arrows.meta,backgrounds,fit}
 \usepackage{amsmath}
 \usetikzlibrary{arrows.meta, bending, calc}
 \usepackage{amssymb,amsfonts,amsmath}
 \usepackage{float}
 \usepackage{booktabs}
 \usepackage{adjustbox}
 \usepackage{amsmath,amssymb}

\usepackage{array}
\usepackage{graphicx}
\usepackage{multirow}
\usepackage{color}
\def\blue{\textcolor{blue}}
\def\red{\textcolor{red}}

\def\blue{\textcolor{blue}}
\usepackage{amsfonts, amsthm, amsmath}
\allowdisplaybreaks[4]
\usepackage{pgfplots}
\usepackage{xcolor}
\usepackage{afterpage}
\usepackage{rotating}

\usepackage{tikz}
\usepackage{lmodern}

\usetikzlibrary{decorations.pathmorphing}
\tikzset{
	redline/.style={red, thick, ->, shorten >=1pt, shorten <=1pt}
}

\usepackage{xcolor}
\usepackage{framed}
\usepackage{graphics}

\usepackage{amssymb}

\usepackage{amscd}

\usepackage{t1enc}

\usepackage[mathscr]{eucal}

\usepackage{indentfirst}

\usepackage{enumitem}

\usepackage{enumitem}

\usepackage{graphicx}

\usepackage{graphics}

\usepackage{pict2e}

\usepackage{mathrsfs}

\usepackage{comment}

\usepackage{enumitem}

\usepackage{enumerate}

\usepackage{hyperref}

\hypersetup{colorlinks=true}

\usepackage{color}
\usepackage{epic}
\usepackage{framed}
\makeatletter
\IfFileExists{mathabx.sty}{}{%
	\expandafter\def\csname ver@mathabx.sty\endcsname{2026/08/02 fallback}%
}
\makeatother
\usepackage{mathabx}
\usepackage{booktabs}
\usetikzlibrary{decorations.pathreplacing,calc}
\usepackage{makecell}
\newcolumntype{V}{!{\vrule width 2pt}}

\numberwithin{equation}{section}

\def\blue{\textcolor{blue}}
\def\red{\textcolor{red}}

\def\blue{\textcolor{blue}}

\theoremstyle{plain}

\newtheorem{theorem}{Theorem}[section]

\newtheorem{conjecture}[theorem]{Conjecture}

\newtheorem{lemma}[theorem]{Lemma}
\newtheorem{definition}[theorem]{Definition}
\newtheorem{example}[theorem]{Example}

\definecolor{handred}{RGB}{220, 0, 0}
\definecolor{handgreen}{RGB}{0, 150, 0}

\def\des{\mathsf{des}}

\def\maj{\mathsf{maj}}

\def\fmaj{\mathsf{fmaj}}
\def\sint{\mathsf{sint}}
\def\nsint{\mathsf{nsint}}

\def\Lint{\mathsf{Lint}}

\def\Leaf{\mathsf{Leaf}}
\def\Rint{\mathsf{Rint}}
\def\NT{\mathsf{NT}}

\def\l{\mathsf{left}}

\def\Sint{\mathsf{Sint}}

\def\l{\mathsf{left}}

\def\lint{\mathsf{lint}}

\def\B{\mathcal{T}}
\def\AT{\mathcal{AT}}

\def\leaf{\mathsf{leaf}}

\def\neg{\mathsf{neg}}

\def\ddes{\mathsf{ddes}}

\begin{document}
	\begin{center}
		{\Large\bf   The $q$-analogues of $\gamma$-positivity of Eulerian polynomials via group actions}
	\end{center}
	
	\begin{center}
		{\small  Taifeng Ding,   Lintong Wang,   Sherry H.F. Yan$^*$}\\[4pt]
		 
	 Department of Mathematics, Zhejiang Normal University\\
		Jinhua 321004, P.R. China
	\end{center}
	\footnote{$^*$Corresponding author. \\{\em E-mail address:}   hfy@zjnu.edu.cn (S.H.F. Yan).}
	
	\vskip0.1in
	\noindent {\bf Abstract.}
	 	Han, Jouhet and Zeng established $q$-analogues of the 
	 $\gamma$-expansion formulas for Eulerian polynomials of types $A$ and $B$.
	 Combinatorial interpretations of the corresponding coefficients
	 $a_{n,k}(q)$ and $b_{n,k}(q)$, however, remained open.
	 In this paper, we provide combinatorial interpretations for these
	 coefficients by using the model of increasing binary trees, thereby
	 resolving a problem posed by Han, Jouhet and Zeng.
	 Our combinatorial approach consists of three main steps:
	 
	 \begin{itemize}
	 	\item construct a Carlitz-type insertion bijection for increasing
	 	binary trees and derive a new combinatorial interpretation of
	 	Carlitz's $q$-Eulerian polynomials of type $A$ in terms of such trees;
	 	
	 	\item introduce a generalized Foata--Strehl action on increasing
	 	binary trees to interpret the coefficients $a_{n,k}(q)$;
	 	
	 	\item derive a new combinatorial interpretation for the $q$-Eulerian
	 	polynomials of type $B$ introduced by Chow and Gessel in terms of
	 	increasing binary trees of type $B$, and develop a generalized
	 	Foata--Strehl action on these trees to interpret the coefficients
	 	$b_{n,k}(q)$.
	 \end{itemize}
	 
	 We further give combinatorial interpretations for the quotients ${a_{n,k}(q)/(-q;q)_{k-1}}$ and 
	 ${b_{n,k}(q)/(1+q)^k(-q;q^2)_k}$ 
	 in terms of Andr\'e trees and a certain class of increasing
	 binary trees of type $B$, respectively. As an application of the latter
	 interpretation, we obtain a combinatorial  interpretation   for   a $q$-analogue of the secant number    and prove the positivity conjecture of Han, Jouhet and
	 Zeng.


	\section{Introduction}

	Let $\mathfrak{S}_{n}$ denote the set of permutations of $[n]:=\{1,2,\ldots, n\}$. A permutation $\pi\in \mathfrak{S}_{n}$ is usually written in one-line notation as $\pi=\pi_1\pi_2\cdots \pi_n$. Given a permutation  $\pi\in \mathfrak{S}_{n}$,  an index $i$, $1\leq i< n$,  is called a  {\em descent  } of $\pi$ if $\pi_i>\pi_{i+1}$.    
	Define  the {\em descent set}  of $\pi$  to be
	$$
	\mathrm{Des}(\pi)=\{i\in[n-1]\mid \pi_i>\pi_{i+1}\}.
	$$
	Let $\mathrm{des(\pi)}=|\mathrm{Des}(\pi)|$.  
	The {\em major index} of $\pi$, denoted by $\mathrm{maj(\pi)}$, is defined to be
	$$
	\mathrm{maj(\pi)}=\sum\limits_{i\in \mathrm{Des(\pi)} } i.
	$$
	For example, if we let  $\pi=475398216$, then $\mathrm{Des}(\pi)=\{2,3,5,6,7\}$, $\mathrm{des}(\pi)=5$ and $\mathrm{maj}(\pi)=2+3+5+6+7=23$.
	The classical {\em Eulerian polynomial} (see~\cite[pp.~32-33]{ST})  $A_n(t)$ is  defined as 
	$$
	A_n(t)=\sum_{\pi\in \mathfrak{S}_{n}}t^{\mathrm{des}(\pi)}.
	$$

	Let $h(t)=\sum_{i=0}^{n}a_it^i$ be a polynomial with nonnegative real coefficients.
	We say that $h(t)$ is {\em $\gamma$-positive} if it can be expanded as
	$$
	h(t)=\sum_{k=0}^{\lfloor {n\over 2}\rfloor} \gamma_k t^k(1+t)^{n-2k}
	$$
	with the {\em$\gamma$-coefficient} $\gamma_k\geq 0$.    
	One of the most remarkable facts about the Eulerian polynomials is the $\gamma$-positivity discovered
	by Foata and Sch\"{u}tzenberger \cite{FS70} which asserts that the Eulerian polynomial $A_n(t)$ has the following $\gamma$-expansion:
	\begin{equation}\label{gamma-A_n(t)}
	A_n(t)=\sum\limits_{k=1}^{\lfloor{(n+1)/ 2}\rfloor}  a_{n,k}t^{k-1}(1+t)^{n+1-2k},
	\end{equation}
	where 
	$$
	a_{n,k}=\# \{\sigma\in \mathfrak{S}_n\mid  \des(\sigma)=k-1,  \ddes(\sigma)=0\}.
	$$
	Here $\ddes(\sigma)$ means the number of double descents of $\sigma$,  that is, the number of indices $i$ ($1\leq i\leq n$ ) such that $\sigma_{i-1}>\sigma_{i}>\sigma_{i+1}$ with the convention that $\sigma_0=\sigma_{n+1}=0$.
	An elegant combinatorial proof of (\ref{gamma-A_n(t)}) via a group action was later constructed by Foata and Strehl~\cite{FS74}.  Since then, $\gamma$-positive  polynomials have  been extensively studied,  see, for example,  \cite{Athanasiadis2018gamma, Branden2008actions,  ChenFuYan2023gessel, JiLin2024binomial, LinMaZhang2021multiperm, LinXuZhao2022multigamma, LinZeng2015gamma, YanHuangYang2022partialgamma, YanYangLin2026bigamma}. 
	Foata and Sch\"{u}tzenberger \cite{FS70}  also proved that for all $n,k\geq 1$, 
	\begin{equation}\label{eq-d_{n,k}}
		a_{n,k}=2^{k-1}d_{n,k},
		\end{equation}
		where  $d_{n,k}$ satisfies the following recurrence relation:
	$$
d_{n,k} = kd_{n-1,k} + (n+2- 2k)d_{n-1,k-1}
	$$
 for $n\geq 2$ and $1\le k\le \lfloor (n+1)/2\rfloor$, with initial condition $d_{1,1}=1$, and $d_{n,k}=0$ whenever $k\le 0$ or $k>\lfloor (n+1)/2\rfloor$.   
	 Actually,  the sum $\sum_k d_{n,k}$ is precisely the {\em Euler number} $E_n$. 
	  Brändén \cite{Branden2008actions} proved a $(p,q)$-analogue of the $\gamma$-expansion formula for Eulerian polynomials and conjectured the divisibility of the $\gamma$-coefficient $\gamma_{n,k}(p,q)$ by $(p+q)^k$.   Shin and Zeng \cite{ShinZeng2010} showed that the fraction $\gamma_{n,k}(p,q)/(p+q)^k$ is a polynomial in $\mathbb N[p,q]$.   A combinatorial interpretation of the  fraction in terms of André permutations  was provided by Pan and Zeng \cite{PanZeng2021}.

	  For $k\geq 1$, let $[k]_q=1+q+q^2+\cdots +q^{k-1}$. Let $(x;q)_n=(1-x)(1-xq)\cdots(1-xq^{n-1})$ and set $(x;q)_0=1$.  
	 The  Carlitz's $q$-Eulerian polynomial of type $A$ \cite{Carlitz} is defined by
	 \begin{equation}\label{eq:qEulerA-def}
	 	A_n(t,q)=\sum_{\sigma\in \mathfrak{S}_n} t^{\des(\sigma)} q^{\maj(\sigma)}
	 	=\sum\limits_{k= 1}^{n}A_{n,k}(q)t^{k-1}
	 \end{equation}
	 whose  factorial generating function is given by 
	  \begin{equation}\label{eq:qEulerA-gf}
	 	\frac{A_n(t,q)}{(t;q)_{n+1}}
	 	=\sum_{k\geq 0} [k+1]_q^n\, t^k.
	 \end{equation}
	   Carlitz~\cite{Carlitz} showed that $A_{n,k}(q)$ satisfies the recurrence relation:
	 \begin{equation}\label{eq:qEulerA-rec}
	 	A_{n,k}(q) = [k]_qA_{n-1,k}(q) + q^{k-1}\,[n+1-k]_q A_{n-1,k-1}(q),\quad 1\le k\le n.
	 \end{equation}

	 Relying on the recurrence relation (\ref{eq:qEulerA-rec}), Han, Jouhet and Zeng \cite{HanJouhetZeng2013qtriangles} derived the following $q$-analogue of (\ref{gamma-A_n(t)}).
	 
	 \begin{theorem}[\cite{HanJouhetZeng2013qtriangles}, Theorem 1]
	 	For $n\geq 1$, there exist polynomials $a_{n,k}(q)\in\mathbb{N}[q]$ such that the $q$-Eulerian polynomials $A_n(t,q)$ can be written as
	 	\begin{equation}\label{eq-Han}
	 		A_n(t,q)=\sum_{k=1}^{\lfloor (n+1)/2\rfloor} a_{n,k}(q)\, t^{k-1}\, (-tq^k;q)_{n+1-2k}.
	 	\end{equation}
	 	Moreover, the polynomials $a_{n,k}(q)$ satisfy the following recurrence relation:
	 	\begin{equation}\label{eq-a_{n,k}(q)}
	 		a_{n,k}(q) = [k]_q\, a_{n-1,k}(q)+ (1+q^{k-1})q^{k-1}\,[n+2-2k]_q\, a_{n-1,k-1}(q)
	 	\end{equation}
	 	for $n\ge 2$ and $1\le k\le \lfloor (n+1)/2\rfloor$, with initial condition $a_{1,1}(q)=1$, and $a_{n,k}(q)=0$ whenever $k\le 0$ or $k>\lfloor (n+1)/2\rfloor$.
	 \end{theorem}
	 
	 Recall that, intuitively, a {\em signed permutation }on $[n]$ can be viewed as an ordinary permutation on $[n]$ with some elements assigned minus signs.   Let $\mathcal{B}_n$ denote the set of signed permutations on $[n]$.   Given a  signed permutation  $\pi\in \mathcal{B}_{n}$,  an index $i$, $0\leq i< n$,  is called a  {\em descent   of type  $B$} of $\pi$ if $\pi_i>\pi_{i+1}$ with the convention $\pi_0=0$.     Let $\des_B(\pi)$ denote the number of descents of type $B$ of $\pi$. 
	 Define the Eulerian polynomial $B_n(t)$ of type $B$ by
	 \begin{equation}\label{eq-B_n(t)}
	 	B_n(t)=\sum_{\pi\in B_n} t^{\des_B(\pi)}
	 	=\sum_{k=0}^{n} B_{n,k}\, t^k,
	 \end{equation}
	 where $B_{n,k}$ is called  the {\em Eulerian number of type $B$} counting the elements of $B_n$ with $k$  descents of type $B$. Following \cite{AdinBrentiRoichman2001hyperoctahedral},  the {\em flag major index} of $\pi\in \mathcal{B}_n$ is defined as
	 $$
	\fmaj(\pi)=2\maj(\pi)+\neg(\pi).
	 $$
	 Here
	 $\neg(\pi)$ means the number of negative entries of $\pi$ and
	 $$
	 \maj(\pi)=\sum_{\substack{1\leq i<n\\ \pi_i>\pi_{i+1}}} i.
	 $$
	 The $q$-Eulerian polynomial of type $B$ \cite{ChowGessel2007hyperoctahedral} is defined by
	 \begin{equation}\label{eq:qEulerA-def}
	 	B_n(t,q)=\sum_{\sigma\in \mathcal{B}_n} t^{\des_B(\sigma)} q^{\fmaj(\sigma)}
	 	=\sum\limits_{k= 0}^{n}B_{n,k}(q)t^{k}
	 \end{equation}
	 whose  factorial generating function is given by 
	\begin{equation}\label{eq-gen-B}
		\frac{B_n(t,q)}{(t;q^2)_{n+1}}=\sum_{k\ge 0}[2k+1]_q^n\, t^k.
	\end{equation}
	The coefficients $B_{n,k}(q)$ satisfy the recurrence relation 
	\begin{equation}\label{eq-re-B_{n,k}}
		B_{n,k}(q)=[2k+1]_q\, B_{n-1,k}(q)+ q^{2k-1}\,[2n-2k+1]_q\, B_{n-1,k-1}(q),\quad 1\le k\le n.
	\end{equation}
	Using the recurrence relation (\ref{eq-re-B_{n,k}}), Han, Jouhet and Zeng \cite{HanJouhetZeng2013qtriangles} obtained the following expansion for $B_n(t,q)$.
	 \begin{theorem}[\cite{HanJouhetZeng2013qtriangles}, Theorem 5] \label{thm-Han-2}
	 	For $n\geq 1$, there exist polynomials $b_{n,k}(q)\in\mathbb{N}[q]$ such that 
	 	\begin{equation}\label{eq-Han-B}
	 		B_n(t,q)
	 		=\sum_{k=0}^{\lfloor n/2\rfloor} b_{n,k}(q)\, t^k\, \bigl(-tq^{2k+1};q^2\bigr)_{n-2k}.
	 	\end{equation}
	 	Moreover, the coefficients $b_{n,k}(q)$ satisfy  the following recurrence relation:
	 	\begin{equation}\label{eq-re-2}
	 		b_{n,k}(q) = [2k+1]_q\, b_{n-1,k}(q)
	 		+ (1+q)\,q^{2k-1}(1+q^{2k-1})\,[n+1-2k]_{q^2}\, b_{n-1,k-1}(q)
	 	\end{equation}
	 	for $n\ge 2$ and $0\le k\le \lfloor n/2\rfloor$, with initial condition $b_{1,0}(q)=1$, and $b_{n,k}(q)=0$ whenever $k<0$ or $k>\lfloor n/2\rfloor$.
	 \end{theorem}
	 It should be mentioned that setting $q=1$ in Theorem \ref{thm-Han-2} recovers the $\gamma$-positivity of the Eulerian polynomial  $B_n(t)$ \cite{Chow2008expansions, Petersen2007peakalgebras, Stembridge2008coxetercones}.

	   The study of the theme of  $\gamma$-positivity from a purely combinatorial aspect was proposed by Athanasiadis \cite{Athanasiadis2018gamma}. 
	     Ordered labeled trees have long been   powerful   tools for combinatorially interpreting the $\gamma$-coefficients of enumerative polynomials; see, for instance, \cite{ChenFuYan2023gessel, LinXuZhao2022multigamma, YanHuangYang2022partialgamma, YanYangLin2026bigamma}.     
	  The main objective of this paper is 
  to provide combinatorial interpretations for the coeﬃcients $a_{n,k}(q)$ and $b_{n,k}(q)$ via group actions  on   increasing binary trees, thereby resolving an open problem posed by Han, Jouhet and Zeng 
	 \cite[Problem 11]{HanJouhetZeng2013qtriangles}. We further give combinatorial interpretations for the quotients ${a_{n,k}(q)/(-q;q)_{k-1}}$ and 
	 ${b_{n,k}(q)/(1+q)^k(-q;q^2)_k}$ 
	 in terms of Andr\'e trees and a certain class of increasing
	 binary trees of type $B$, respectively. As an application of the latter
	 interpretation, we obtain a combinatorial  interpretation   for   a $q$-analogue of the secant number    and prove the positivity conjecture of Han, Jouhet and
	 Zeng \cite[Conjecture 8]{HanJouhetZeng2013qtriangles}.

	 \section{Main results}
	 
	 In this section, we shall state our main results whose proofs will be presented in the
	 subsequent sections.
	 
	 Let us first review some terminology related to trees. An ordered tree is a tree equipped with a distinguished node called the root, such that the subtrees of every node are linearly ordered. In this paper, all trees are assumed to be ordered.
	Recall that  a  \blue{\em binary tree} is an ordered  tree where each internal node has either a left child, a right child, or both.   An internal node of a binary tree is said to be a {\em singleton internal node} if it has only one child.

	\begin{definition}[Increasing binary trees]
		  An  \blue{  increasing binary  tree} on $[n]$ is a labeled  binary tree   such that
		\begin{itemize} 
			\item  the labels of the nodes form precisely the set $[n]$ and
			\item  the labels along a path from the root to any leaf are  increasing.
		\end{itemize}
		Denote by $\B_n$ the set of all   increasing binary trees on $[n]$.  See Figure~\ref{fig:alpha} for an illustration of increasing binary trees.
	\end{definition}

 Given a binary tree $T\in \B_n$,  let $\Leaf(T)$ and $\Sint(T)$ denote the set of leaves and the set of singleton internal nodes of $T$, respectively. 
 Let $\leaf(T)$ and $\l(T)$ denote the number of leaves and the number of left edges of $T$, respectively. 
 The statistic $\alpha(T)$ is defined recursively as follows:
If $T$ has only one node, set $\alpha(T)=0$. Otherwise, define $\hat{T}$ to be the tree obtained from $T$ by removing  node $n$  together with  its incident edge.  Suppose that $\hat{T}$ has exactly $k$ leaves which are listed in increasing order:
$$
\Leaf(\hat{T})=\{\ell_1<\ell_2<\cdots<\ell_k\}
$$
  and has exactly $m$ singleton internal nodes which are listed in increasing order:
  $$
  \Sint(\hat{T})=\{u_1<u_2<\cdots<u_m\}.
  $$
\begin{itemize}
 
	\item If node $n$ is the child of node $\ell_i$ for some $1\leq i\leq k$, then set $\alpha(T)=\alpha(\hat{T})+i-1$.
	\item If node $n$ is the left  child of node $u_i$ for some $1\leq i\leq m$, then set $\alpha(T)=\alpha(\hat{T})+\leaf(\hat{T})+m-i$.
		\item If node $n$ is the right  child of node $u_i$ for some $1\leq i\leq m$, then set $\alpha(T)=\alpha(\hat{T})+2\leaf(\hat{T})+i-1$.
\end{itemize}

Figure \ref{fig:alpha} illustrates the successive recursive computation 
 of $\alpha(T_7)$, where $\hat{T}_i=T_{i-1}$ for all $1\leq i\leq 7$.

An  internal node of a binary tree  is said to be {\em left (resp.,  right)} if it has only  a left (resp.,  right) child.  
  Given a binary tree $T\in \B_n$,   let $\Sint(T)=\{v_1< v_2<\ldots< v_m\}$.   Define  $$\beta(T)=(\beta_1, \beta_2, \ldots, \beta_m)$$
  where 
  $$
  \beta_i=\left\{
   \begin{array}{ll}
  	1 & \text{if }  v_i\in \Lint(T),\\[2pt]
  0& \text{otherwise}.
  \end{array}
  \right.
  $$
  Here $\Lint(T)$ denotes the set of  left  internal nodes   of $T$. 
  Let $T_5$ be the tree displayed in Figure \ref{fig:alpha}. Then we have
  $$
  \beta(T_5)=(\beta_1, \beta_2, \beta_3)=(1,0,0).
  $$ 
The {\em major index} of an increasing binary tree $T$ is defined as 
	 $$
	 \maj(T)=\alpha(T)+\sum\limits_{i=1}^{m}(\leaf(T)+i-1) \beta_i.
	 $$
	 Continuing with our running example, we have
	 $$
	 \begin{array}{lll}
	 \maj(T_5)&=&\alpha(T_5)+\sum\limits_{i=1}^{3}(\leaf(T_5)+i-1)\beta_i\\
	 	&=& 5+(2+1-1)=7.
	 	\end{array}
	 $$
	Define $$\B^*_{n,k}=\{T\in \B_n\mid  \leaf(T)=k, \,\, \lint(T)=0\},$$
	 where $\lint(T)$ denotes the number of left   internal nodes of $T$.  
	For example, Figure~\ref{fig:three-labeled-trees} shows the two trees $T_1$ and $T_2$ belonging to $\B^*_{3,2}$, as well as the unique tree $T_3$ of $\B^*_{3,1}$.

	\begin{figure}[htbp]
		\centering
		\footnotesize
		\setlength{\tabcolsep}{4pt}
		\renewcommand{\arraystretch}{1.15}
		\setlength{\arrayrulewidth}{0.35pt}
		
		\tikzset{
			v/.style={circle,fill,inner sep=1.05pt}
		}
		
		\begin{tabular}{c@{\hspace{0.4em}}|@{\hspace{0.7em}}c|c|c|c|c@{\hspace{0.4em}}|@{\hspace{0.7em}}c|c|c|c}
			\hline
			$i$ & $T_i$ & $\operatorname{Leaf}(T_i)$ & $\operatorname{Sint}(T_i)$ & $\alpha(T_i)$
			& $i$ & $T_i$ & $\operatorname{Leaf}(T_i)$ & $\operatorname{Sint}(T_i)$ & $\alpha(T_i)$ \\
			\hline
			
			0 &
			\begin{tikzpicture}[baseline=-2pt,scale=.52]
				\node[v,label=above:$1$] at (0,0) {};
			\end{tikzpicture}
			&
			$\{1\}$ & $\varnothing$ & $0$
			&
			4 &
			\begin{tikzpicture}[baseline=-2pt,scale=.52]
				\node[v,label=above:$1$] (1) at (0,0) {};
				\node[v,label=left:$2$] (2) at (-1,-1) {};
				\node[v,label=right:$3$] (3) at (1,-1) {};
				\node[v,label=left:$4$] (4) at (.4,-2) {};
				\node[v,label=right:$5$] (5) at (1,-3) {};
				\draw (1)--(2) (1)--(3) (3)--(4) (4)--(5);
			\end{tikzpicture}
			&
			$\{2,5\}$ & $\{3,4\}$ & $4$
			\\
			\hline
			
			1 &
			\begin{tikzpicture}[baseline=-2pt,scale=.52]
				\node[v,label=above:$1$] (1) at (0,0) {};
				\node[v,label=left:$2$] (2) at (-1,-1) {};
				\draw (1)--(2);
			\end{tikzpicture}
			&
			$\{2\}$ & $\{1\}$ & $0$
			&
			5 &
			\begin{tikzpicture}[baseline=-2pt,scale=.52]
				\node[v,label=above:$1$] (1) at (0,0) {};
				\node[v,label=left:$2$] (2) at (-1,-1) {};
				\node[v,label=right:$3$] (3) at (1,-1) {};
				\node[v,label=left:$4$] (4) at (.4,-2) {};
				\node[v,label=right:$5$] (5) at (1,-3) {};
				\node[v,label=right:$6$] (6) at (1.6,-4) {};
				\draw (1)--(2) (1)--(3) (3)--(4) (4)--(5) (5)--(6);
			\end{tikzpicture}
			&
			$\{2,6\}$ & $\{3,4,5\}$ & $5$
			\\
			\hline
			
			2 &
			\begin{tikzpicture}[baseline=-2pt,scale=.52]
				\node[v,label=above:$1$] (1) at (0,0) {};
				\node[v,label=left:$2$] (2) at (-1,-1) {};
				\node[v,label=right:$3$] (3) at (1,-1) {};
				\draw (1)--(2) (1)--(3);
			\end{tikzpicture}
			&
			$\{2,3\}$ & $\varnothing$ & $2$
			&
			6 &
			\begin{tikzpicture}[baseline=-2pt,scale=.52]
				\node[v,label=above:$1$] (1) at (0,0) {};
				\node[v,label=left:$2$] (2) at (-1,-1) {};
				\node[v,label=right:$3$] (3) at (1,-1) {};
				\node[v,label=left:$4$] (4) at (.4,-2) {};
				\node[v,label=left:$7$] (7) at (-.3,-3) {};
				\node[v,label=right:$5$] (5) at (1,-3) {};
				\node[v,label=right:$6$] (6) at (1.6,-4) {};
				\draw (1)--(2) (1)--(3) (3)--(4) (4)--(7) (4)--(5) (5)--(6);
			\end{tikzpicture}
			&
			$\{2,6,7\}$ & $\{3,5\}$ & $8$
			\\
			\hline
			
			3 &
			\begin{tikzpicture}[baseline=-2pt,scale=.52]
				\node[v,label=above:$1$] (1) at (0,0) {};
				\node[v,label=left:$2$] (2) at (-1,-1) {};
				\node[v,label=right:$3$] (3) at (1,-1) {};
				\node[v,label=left:$4$] (4) at (.4,-2) {};
				\draw (1)--(2) (1)--(3) (3)--(4);
			\end{tikzpicture}
			&
			$\{2,4\}$ & $\{3\}$ & $3$
			&
			7 &
			\begin{tikzpicture}[baseline=-2pt,scale=.52]
				\node[v,label=above:$1$] (1) at (0,0) {};
				\node[v,label=left:$2$] (2) at (-1,-1) {};
				\node[v,label=right:$3$] (3) at (1,-1) {};
				\node[v,label=left:$4$] (4) at (.4,-2) {};
				\node[v,label=right:$8$] (8) at (1.8,-2) {};
				\node[v,label=left:$7$] (7) at (-.3,-3) {};
				\node[v,label=right:$5$] (5) at (1,-3) {};
				\node[v,label=right:$6$] (6) at (1.6,-4) {};
				\draw (1)--(2) (1)--(3) (3)--(4) (3)--(8) (4)--(7) (4)--(5) (5)--(6);
			\end{tikzpicture}
			&
			$\{2,6,7,8\}$ & $\{5\}$ & $14$
			\\
			\hline
			
		\end{tabular}
		
		\caption{The recursive  computation of $\alpha(T_7)$.}\label{fig:alpha}
		\label{fig:alpha-example}
	\end{figure}

	 	\begin{figure}[htbp]
	 	\centering
	 	\begin{tikzpicture}[
	 		vertex/.style={
	 			circle,
	 			draw,
	 			fill=black,
	 			minimum size=2pt,
	 			inner sep=0pt,
	 			line width=0.7pt
	 		},
	 		edge/.style={line width=0.3pt}
	 		]
	 		
	 		\begin{scope}[shift={(0,0)}]
	 			\node[vertex,label=above:{$1$}] (a1) at (0,1.4) {};
	 			\node[vertex,label=below left:{$2$}] (a2) at (-0.8,0) {};
	 			\node[vertex,label=below right:{$3$}] (a3) at (0.8,0) {};
	 			
	 			\draw[edge] (a1) -- (a2);
	 			\draw[edge] (a1) -- (a3);
	 			
	 			\node at (0,-0.85) {$T_1$};
	 		\end{scope}
	 		
	 		\begin{scope}[shift={(4.2,0)}]
	 			\node[vertex,label=above:{$1$}] (b1) at (0,1.4) {};
	 			\node[vertex,label=below left:{$3$}] (b3) at (-0.8,0) {};
	 			\node[vertex,label=below right:{$2$}] (b2) at (0.8,0) {};
	 			
	 			\draw[edge] (b1) -- (b3);
	 			\draw[edge] (b1) -- (b2);
	 			
	 			\node at (0,-0.85) {$T_2$};
	 		\end{scope}
	 		
	 		\begin{scope}[shift={(8.4,0)}]
	 			\node[vertex,label=above:{$1$}] (c1) at (-0.8,1.6) {};
	 			\node[vertex,label=above right:{$2$}] (c2) at (0,0.8) {};
	 			\node[vertex,label=below right:{$3$}] (c3) at (0.8,0) {};
	 			
	 			\draw[edge] (c1) -- (c2) -- (c3);
	 			
	 			\node at (0,-0.85) {$T_3$};
	 		\end{scope}
	 		
	 	\end{tikzpicture}
	 	\caption{Three increasing binary trees.}
	 	\label{fig:three-labeled-trees}
	 \end{figure}
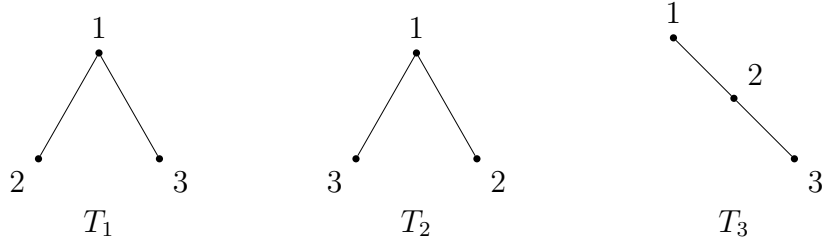

Now we are in the position to state our first main result, which provides   combinatorial
interpretations for  $A_n(t,q)$ and $a_{n,k}(q)$.
	 \begin{theorem}\label{thm-gamma}
	 	For $n\geq 1$, 
	 	we have 
	 	\begin{equation}\label{eq-gamma-1}
	 		A_n(t,q)=\sum\limits_{T\in \B_n}t^{\l(T)}q^{\maj(T)}=\sum\limits_{k=1}^{\lfloor {(n+1)/  2}\rfloor}  
	 		a_{n, k}(q)t^{k-1}(-tq^k; q)_{n+1-2k}, 
	 		\end{equation}
	 		where 
	 		$$
	 		a_{n,k}(q)=\sum_{T\in \B^*_{n,k}}q^{\alpha(T)}.
	 		$$
	 	\end{theorem}
	 	
	 	 \begin{example}
	 		Take $n=3$ in Theorem~\ref{thm-gamma}. The expansion of $A_3(t,q)$ is given by
	 		\[
	 		A_3(t,q)=t(q+q^2)+(1+tq)(1+tq^2)=t\big(q^{\alpha(T_2)}+q^{\alpha(T_1)}\big)+ q^{\alpha(T_3)}(1+tq)(1+tq^2),
	 		\]
	 		where $T_1$, $T_2$, and $T_3$ are shown in Figure~\ref{fig:three-labeled-trees} with $\alpha(T_1)=2$, $\alpha(T_2)=1$, and $\alpha(T_3)=0$.
	 	\end{example}
	 		\begin{definition}[Andr\'e  trees]
	 		A binary tree $T\in\B_n$ is called an \blue{ Andr\'e  tree} on $[n]$ if the smallest
	 		child of any internal node is its right child. 
	 		Denote by $\AT_n$ the set of all   Andr\'e   trees on $[n]$. 
	 	\end{definition} 		
	 	Andr\'e trees are in natural bijection with Andr\'e permutations \cite{Lin-Wen-Yan}, the latter of which were introduced by Foata and Sch\"utzenberger \cite{FS70} and further studied by Strehl \cite{Str74} and Foata and Strehl \cite{FS74,FS76}.
	 	
	 	Define  $\AT_{n,k}$ to be the set of trees  $T\in \AT_n$   with $\leaf(T)$=k.  For example, Figure~\ref{fig:three-labeled-trees-Andre} shows the four   trees   belonging to $\mathcal{AT}_{4,2}$.
	 		\begin{figure}[htbp]
	 		\centering
	 		
	 		\begin{tikzpicture}[
	 			vertex/.style={
	 				circle,
	 				draw=black,
	 				fill=black,
	 				inner sep=1pt,
	 				minimum size=3pt
	 			},
	 			edge/.style={
	 				line width=0.6pt
	 			}
	 			]

	 			\begin{scope}[xshift=0cm]
	 				
	 				\node[vertex] (t1r) at (0,1.2){};
	 				\node[vertex] (t13) at (-0.7,0.3){};
	 				\node[vertex] (t12) at (0.7,0.3){};
	 				\node[vertex] (t14) at (1.3,-0.6){};
	 				
	 				\draw[edge](t1r)--(t13);
	 				\draw[edge](t1r)--(t12);
	 				\draw[edge](t12)--(t14);
	 				
	 				\node[above] at (t1r){$1$};
	 				\node[left] at (t13){$3$};
	 				\node[right] at (t12){$2$};
	 				\node[right] at (t14){$4$};
	 				
	 				\node at (0,-1.5){$T_1$};
	 				
	 			\end{scope}

	 			\begin{scope}[xshift=3.2cm]
	 				
	 				\node[vertex] (t2r) at (0,1.2) {};
	 				\node[vertex] (t23) at (-0.7,0.3) {};
	 				\node[vertex] (t22) at (0.7,0.3) {};
	 				\node[vertex] (t24) at (0,-0.6) {};
	 				
	 				\draw[edge] (t2r)--(t23);
	 				\draw[edge] (t2r)--(t22);
	 				\draw[edge] (t23)--(t24);
	 				
	 				\node[above=1pt] at (t2r) {$1$};
	 				\node[left=2pt] at (t23) {$3$};
	 				\node[right=2pt] at (t22) {$2$};
	 				\node[below=2pt] at (t24) {$4$};
	 				
	 				\node at (0,-1.5) {$T_2$};
	 				
	 			\end{scope}

	 			\begin{scope}[xshift=6.4cm]
	 				
	 				\node[vertex] (t3r) at (0,1.2){};
	 				\node[vertex] (t34) at (-0.7,0.3){};
	 				\node[vertex] (t32) at (0.7,0.3){};
	 				\node[vertex] (t33) at (1.3,-0.6){};
	 				
	 				\draw[edge](t3r)--(t34);
	 				\draw[edge](t3r)--(t32);
	 				\draw[edge](t32)--(t33);
	 				
	 				\node[above] at (t3r){$1$};
	 				\node[left] at (t34){$4$};
	 				\node[right] at (t32){$2$};
	 				\node[right] at (t33){$3$};
	 				
	 				\node at (0,-1.5){$T_3$};
	 				
	 			\end{scope}

	 			\begin{scope}[xshift=9.6cm]
	 				
	 				\node[vertex] (t4r) at (-0.3,1.2){};
	 				\node[vertex] (t42) at (0.35,0.45){};
	 				\node[vertex] (t44) at (-0.35,-0.25){};
	 				\node[vertex] (t43) at (1.05,-0.25){};
	 				
	 				\draw[edge](t4r)--(t42);
	 				\draw[edge](t42)--(t44);
	 				\draw[edge](t42)--(t43);
	 				
	 				\node[above] at (t4r){$1$};
	 				\node[right] at (t42){$2$};
	 				\node[left] at (t44){$4$};
	 				\node[right] at (t43){$3$};
	 				
	 				\node at (0.35,-1.5){$T_4$};
	 				
	 			\end{scope}

	 		\end{tikzpicture}
	 		\caption{Four Andr\'e trees  in $\AT_{4,2}$.}
	 		\label{fig:three-labeled-trees-Andre}
	 	\end{figure}
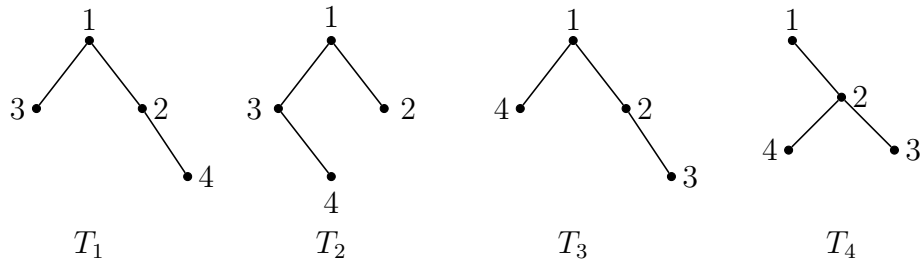

	 	Our second  result is  concerned with the $q$-analogue of (\ref{eq-d_{n,k}}), which provides a combinatorial interpretation of the quotient $ {a_{n,k}(q)/(-q; q)_{k-1}}$ in terms of  enumerative polynomials on Andr\'e trees.
	 	\begin{theorem}\label{thm-d_{n,k}(q)}
	 		For    $n,k\geq 1$, we have
	 		$$
	 		a_{n,k}(q)=(-q; q)_{k-1}d_{n,k}(q) ,
	 		$$
	 		where 
	 		$$
	 		d_{n,k}(q)= \sum_{T\in \AT_{n,k}}q^{\alpha(T)}.
	 		$$
	 		 
	 		\end{theorem}

	 	 	Let $T\in \B_n$ be a binary tree.  A \blue{\em right chain}  in $T$ is a maximal path consisting of only right  edges.  The right chain containing the root of $T$ is called the {\em right arm} of $T$.  Take the tree $T^{(7)}$ displayed in Figure~\ref{fig:alpha} as an example. Then the right arm of $T^{(7)}$ is given by $1-3-8$.
	 	 	
	 	 		\begin{definition}[Increasing binary trees of type $B$]
	 	
	 	 		An \blue{increasing binary tree of type $B$} on $[n]$ is a labeled binary tree with $n+1$ nodes satisfying the following conditions:
	 	 		\begin{itemize}
	 	 			\item  The node labeled by $\infty$ is a leaf  and is the rightmost node on the right arm;
	 	 			\item The absolute values of the labels of all other nodes form precisely the set  $[n]$;
	 	 			\item The absolute values of the labels are strictly increasing along every path from the root to a leaf.
	 	 		\end{itemize}

	 	 		Denote by $\widetilde{\B}_n$ the set of all   increasing binary trees  of
	 	 		 type $B$ on $[n]$. Figure~\ref{fig:tree-B} depicts an increasing binary tree of type $B$ on $[9]$.
	 	 	\end{definition}
	 	
	 	\begin{figure}[htbp]
	 		\centering
	 		\begin{tikzpicture}[
	 				vertex/.style={
	 				circle,
	 				draw=black,
	 				fill=black,
	 				inner sep=1pt,
	 				minimum size=3pt
	 			},
	 			edge/.style={
	 				line width=0.6pt
	 			}
	 			]
	 			\node[vertex,label=above:{$-1$}] (v1) at (0,4.8) {};
	 			\node[vertex,label=below left:{$2$}] (v2) at (-1.8,3.5) {};
	 			\node[vertex,label=above right:{$-3$}] (v3) at (1.5,3.5) {};
	 			
	 			\node[vertex,label=left:{$-4$}] (v4) at (-0.1,2.15) {};
	 			\node[vertex,label=above right:{$-8$}] (v8) at (3.15,2.15) {};
	 			
	 			\node[vertex,label=below left:{$-7$}] (v7) at (-1.5,0.75) {};
	 			\node[vertex,label=left:{$5$}] (v5) at (1.0,0.75) {};
	 			\node[vertex,label=below left:{$9$}] (v9) at (2.25,0.75) {};
	 			\node[vertex,label=below right:{$\infty$}] (vinf) at (4.35,0.75) {};
	 			
	 			\node[vertex,label=below right:{$6$}] (v6) at (2.1,-0.65) {};
	 			
	 			\draw[edge]
	 			(v1) -- (v2)
	 			(v1) -- (v3)
	 			(v3) -- (v4)
	 			(v3) -- (v8)
	 			(v4) -- (v7)
	 			(v4) -- (v5)
	 			(v5) -- (v6)
	 			(v8) -- (v9)
	 			(v8) -- (vinf);
	 		\end{tikzpicture}
	 		\caption{ A tree  $T\in \widetilde{\B}_9$. }\label{fig:tree-B}
	 	\end{figure}

	 	Given a  tree $T\in \widetilde{\B}_n$, the statistic $\alpha_B(T)$ is defined recursively as follows:
	 	If $T$ has only one node, set $\alpha_B(T)=0$.  Otherwise,  construct a tree $T'\in  \widetilde{\B}_{n-1}$ by the following procedure.
	 	\begin{itemize}
	 		\item If node $n$ (resp., $-n$) is a leaf with parent $u$, define $T'$ as the tree obtained from $T$ by deleting node $n$ (resp., $-n$) together with its incident edge.  If $u$ is a singleton internal node of $T$, further relabel  $u$ as $|u|$ (resp., $-|u|$).
	 		
	 		\item If  node $\pm n$ is the parent of  node $\infty$, then  let $T'$ be the tree obtained from $T$ by removing  node  $\infty$ together with its incident edge,  and relabeling  node  $\pm n$ by $\infty$.
	 	\end{itemize}

	 Suppose that $T'$ has exactly $k+1$ leaves  which are listed in increasing order: $$\Leaf(T')=\{\ell_1<\ell_2< \ldots <\ell_k<\infty\}, $$ with the convention that
	 $$
	 -n<\cdots <-1<1<\cdots<n<\infty, 
	 $$ and has exactly $m$ singleton internal nodes which are listed in increasing order of their absolute values: $$\Sint(T')=\{u_1,u_2,\ldots,u_m\}, $$ with $ |u_1|<|u_2|<\cdots<|u_m|.$
	 
	 Define $\alpha_B(T)$ as follows:
	 	\begin{itemize}
	 		
	 		\item  If  the node  $\pm n$ is the left  child  of the  node  with absolute value $|\ell_i|$ for some $1\leq i\leq k$,  then set $\alpha_B(T)=\alpha_B(T')+2i-2$.
	 		
	 		\item If  the node $\pm n$ is the right  child  of the  node  with absolute value $|\ell_i|$ for some $1\leq i\leq k$,   then set $\alpha_B(T)=\alpha_B(T')+2i-1$.
	 		
	 		\item If  node  $\pm n$ is the parent of  node  $\infty$, then set $\alpha_B(T)=\alpha_B(T')+ 2\leaf(T')-2$.
	 	    
	 	    \item If   node $n$ is  a child of  node $u_i$ for some $1\leq i\leq m$ and   $u_i$ is positive, then set $\alpha_B(T)=\alpha_B(T')+2\leaf(T')-1+2(m-i)$.
	 		
	 		\item If node  $-n$ is  a child of  node   $u_i$ for some $1\leq i\leq m$ and   $u_i$ is positive, then set $\alpha_B(T)=\alpha_B(T')+ 2\leaf(T')+2(m-i)$.
	 		
	 		\item If   node $n$ is  a child of   node $u_i$ for some $1\leq i\leq m$ and  $u_i$ is negative, then set $\alpha_B(T)=\alpha_B(T')+2(2\leaf(T')-1)+2(i-1)$.
	 	    
	 	    \item If    node $-n$ is  a child of  node  $u_i$ for some $1\leq i\leq m$ and $u_i$ is negative,  then set $\alpha_B(T)=\alpha_B(T')+2(2\leaf(T')-1)+2(i-1)+1$.
	 		
	 	\end{itemize}
	 See Figure \ref{fig:alpha_B} for an illustration of the successive recursive computation for $\alpha_B(T^{(9)})$, where $(T^{(i)})'=T^{(i-1)}$ for all $1\leq i\leq 9$.

	 Given a tree $T\in \widetilde{\B}_n$, let $\Sint(T)=\{v_1, v_2, \ldots, v_{m}\}$ with $|v_1|<|v_2|<\cdots<|v_m|$. 
	  Define  $\epsilon(T)=(\epsilon_1, \epsilon_2, \ldots, \epsilon_m)$ where 
	 	$$
	 	\epsilon_i=\left\{
	 	\begin{array}{ll}
	 		1 & \text{if  $v_i$ is negative} \\[2pt]
	 		0& \text{otherwise}.
	 	\end{array}
	 	\right.
	 	$$
	 	The {\em flag major index} of an increasing binary tree $T$ of type $B$  is defined as 
	 	$$
	 	\fmaj(T)=\alpha_B(T)+\sum\limits_{i=1}^{m}(2\leaf(T)+2i-3) \epsilon_i.
	 	$$
	 	For example, let $T^{(6)}$ be the tree  shown in Figure \ref{fig:alpha_B}. Then we have $\leaf(T^{(6)})=2$, 
	 	$\Sint(T^{(6)})=\{v_1, v_2, v_3, v_4\}=\{-2, 3,4,-5\}$ and 
	 	$\epsilon(T^{(6)})=(1,0,0,1)$.  Thus, the flag major index of $T^{(6)}$ is given by
	 	$$
	 	\begin{array}{lll}
	 	\fmaj(T^{(6)})&=&\alpha_B(T^{(6)})+\sum\limits_{i=1}^{4}(2\leaf(T^{(6)})+2i-3) \epsilon_i\\
	 	&=& 5+ 3+9\\
	 	&=& 17.
	 	\end{array}
	 	$$
	 	Let  $\nsint(T)$ denote the number of negative   singleton internal nodes  of $T$. 
	 	The {\em weight} of an increasing binary tree $T$ of type $B$  is defined as 
	 	$$
	 	\omega(T)=\leaf(T)-1+\nsint(T).
	 	$$
	 	Define 
	 	$$\widetilde{\B}^*_{n,k}=\{T\in\widetilde{\B}_n\mid  \leaf(T)=k+1, \,\, \nsint(T)=0\}.$$
	 		For example, Figure~\ref{fig:three-labeled-trees-B} shows the  four trees $T_1$, $T_2$, $T_3$ and $T_4$ belonging to $\widetilde{\B}^*_{2,1}$, as well as the unique tree $T_5$ of $\widetilde{\B}^*_{2,0}$.

	 	 \tikzset{
	 		typeB edge/.style={line width=0.45pt},
	 		typeB vertex/.style={circle,fill=black,inner sep=1.1pt}
	 	}
	 	
	 	\newcommand{\TypeBTreePic}[1]{%
	 		\begin{tikzpicture}[x=0.48cm,y=0.44cm,
	 			baseline=(current bounding box.center),
	 			every label/.style={inner sep=0.3pt}]
	 			\path[use as bounding box] (-2.55,-4.30) rectangle (2.10,0.90);
	 			\ifnum#1=0
	 			\node[typeB vertex,label={[font=\scriptsize]above:$\infty$}]
	 			(inf) at (0,0) {};
	 			\else
	 			\node[typeB vertex,label={[font=\scriptsize]above:$-1$}]
	 			(n1) at (0,0) {};
	 			\node[typeB vertex,label={[font=\scriptsize]right:$\infty$}]
	 			(inf) at (1.20,-0.65) {};
	 			\draw[typeB edge] (n1)--(inf);
	 			\ifnum#1>1
	 			\node[typeB vertex,label={[font=\scriptsize]left:$-2$}]
	 			(n2) at (-0.85,-0.65) {};
	 			\draw[typeB edge] (n1)--(n2);
	 			\fi
	 			\ifnum#1>2
	 			\ifnum#1=3
	 			\node[typeB vertex,label={[font=\scriptsize]left:$-3$}]
	 			(n3) at (-1.20,-1.35) {};
	 			\else
	 			\node[typeB vertex,label={[font=\scriptsize]left:$3$}]
	 			(n3) at (-1.20,-1.35) {};
	 			\fi
	 			\draw[typeB edge] (n2)--(n3);
	 			\fi
	 			\ifnum#1>3
	 			\ifnum#1=4
	 			\node[typeB vertex,label={[font=\scriptsize]right:$-4$}]
	 			(n4) at (-0.80,-2.05) {};
	 			\else
	 			\node[typeB vertex,label={[font=\scriptsize]right:$4$}]
	 			(n4) at (-0.80,-2.05) {};
	 			\fi
	 			\draw[typeB edge] (n3)--(n4);
	 			\fi
	 			\ifnum#1>4
	 			\node[typeB vertex,label={[font=\scriptsize]left:$-5$}]
	 			(n5) at (-1.20,-2.75) {};
	 			\draw[typeB edge] (n4)--(n5);
	 			\fi
	 			\ifnum#1>5
	 			\node[typeB vertex,label={[font=\scriptsize]right:$-6$}]
	 			(n6) at (-0.80,-3.45) {};
	 			\draw[typeB edge] (n5)--(n6);
	 			\fi
	 			\ifnum#1>6
	 			\node[typeB vertex,label={[font=\scriptsize]right:$7$}]
	 			(n7) at (-0.30,-1.35) {};
	 			\draw[typeB edge] (n2)--(n7);
	 			\fi
	 			\ifnum#1>7
	 			\node[typeB vertex,label={[font=\scriptsize]left:$-8$}]
	 			(n8) at (-1.65,-2.05) {};
	 			\draw[typeB edge] (n3)--(n8);
	 			\fi
	 			\ifnum#1>8
	 			\node[typeB vertex,label={[font=\scriptsize]right:$9$}]
	 			(n9) at (-0.40,-2.75) {};
	 			\draw[typeB edge] (n4)--(n9);
	 			\fi
	 			\fi
	 		\end{tikzpicture}%
	 	}
	 	
	 	\begin{figure}[htbp]
	 		\centering
	 		\begingroup
	 		\setlength{\tabcolsep}{2.2pt}
	 		\renewcommand{\arraystretch}{1.02}
	 		\footnotesize
	 		\resizebox{\linewidth}{!}{%
	 			\begin{tabular}{c|c|c|c|c||c|c|c|c|c}
	 				\hline
	 				$i$ & $T^{(i)}$ & $\Leaf(T^{(i)})$ & $\Sint(T^{(i)})$
	 				& $\alpha_B(T^{(i)})$
	 				& $i$ & $T^{(i)}$ & $\Leaf(T^{(i)})$ & $\Sint(T^{(i)})$
	 				& $\alpha_B(T^{(i)})$\\
	 				\hline
	 				0 & \TypeBTreePic{0} & $\{\infty\}$ & $\varnothing$ & 0
	 				& 5 & \TypeBTreePic{5} & $\{-5,\infty\}$
	 				& $\{-2,3,4\}$ & 4\\
	 				\hline
	 				1 & \TypeBTreePic{1} & $\{\infty\}$ & $\{-1\}$ & 0
	 				& 6 & \TypeBTreePic{6} & $\{-6,\infty\}$
	 				& $\{-2,3,4,-5\}$ & 5\\
	 				\hline
	 				2 & \TypeBTreePic{2} & $\{-2,\infty\}$ & $\varnothing$ & 3
	 				& 7 & \TypeBTreePic{7} & $\{-6,7,\infty\}$
	 				& $\{3,4,-5\}$ & 11\\
	 				\hline
	 				3 & \TypeBTreePic{3} & $\{-3,\infty\}$ & $\{-2\}$ & 3
	 				& 8 & \TypeBTreePic{8} & $\{-8,-6,7,\infty\}$
	 				& $\{4,-5\}$ & 21\\
	 				\hline
	 				4 & \TypeBTreePic{4} & $\{-4,\infty\}$ & $\{-2,3\}$ & 4
	 				& 9 & \TypeBTreePic{9} & $\{-8,-6,7,9,\infty\}$
	 				& $\{-5\}$ & 30\\
	 				\hline
	 			\end{tabular}%
	 		}
	 		\endgroup
	 		\caption{The recursive computation of $\alpha_B(T^{(9)})$. }\label{fig:alpha_B}
	 	\end{figure}

	 		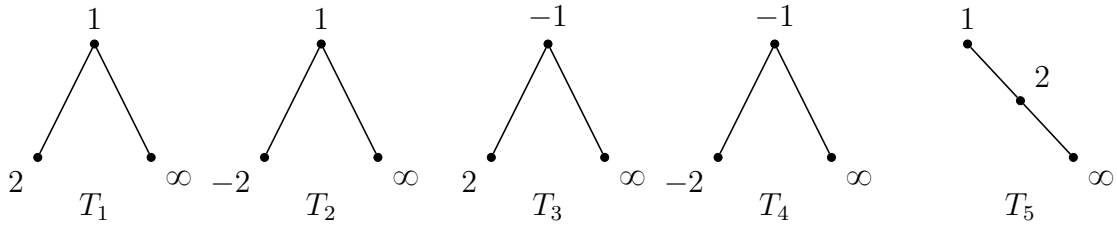
\begin{figure}[htbp]
	 		\centering
	 		\begin{tikzpicture}[
	 				vertex/.style={
	 				circle,
	 				draw=black,
	 				fill=black,
	 				inner sep=1pt,
	 				minimum size=3pt
	 			},
	 			edge/.style={
	 				line width=0.6pt
	 			}
	 			]
	 			edge/.style={line width=0.7pt},
	 			every label/.style={
	 				font=\small,
	 				inner sep=1pt,
	 				label distance=-0.5pt
	 			}
	 			]
	 			
	 			\begin{scope}[xshift=0cm]
	 				\node[vertex,label=above:{$1$}] (r1) at (0,1.5) {};
	 				\node[vertex,label=below left:{$2$}] (a1) at (-0.75,0) {};
	 				\node[vertex,label=below right:{$\infty$}] (b1) at (0.75,0) {};
	 				
	 				\draw[edge] (r1)--(a1) (r1)--(b1);
	 				\node at (0,-0.65) {$T_1$};
	 			\end{scope}
	 			
	 			\begin{scope}[xshift=3cm]
	 				\node[vertex,label=above:{$1$}] (r2) at (0,1.5) {};
	 				\node[vertex,label=below left:{$-2$}] (a2) at (-0.75,0) {};
	 				\node[vertex,label=below right:{$\infty$}] (b2) at (0.75,0) {};
	 				
	 				\draw[edge] (r2)--(a2) (r2)--(b2);
	 				\node at (0,-0.65) {$T_2$};
	 			\end{scope}
	 			
	 			\begin{scope}[xshift=6cm]
	 				\node[vertex,label=above:{$-1$}] (r3) at (0,1.5) {};
	 				\node[vertex,label=below left:{$2$}] (a3) at (-0.75,0) {};
	 				\node[vertex,label=below right:{$\infty$}] (b3) at (0.75,0) {};
	 				
	 				\draw[edge] (r3)--(a3) (r3)--(b3);
	 				\node at (0,-0.65) {$T_3$};
	 			\end{scope}
	 			
	 			\begin{scope}[xshift=9cm]
	 				\node[vertex,label=above:{$-1$}] (r4) at (0,1.5) {};
	 				\node[vertex,label=below left:{$-2$}] (a4) at (-0.75,0) {};
	 				\node[vertex,label=below right:{$\infty$}] (b4) at (0.75,0) {};
	 				
	 				\draw[edge] (r4)--(a4) (r4)--(b4);
	 				\node at (0,-0.65) {$T_4$};
	 			\end{scope}
	 			
	 			\begin{scope}[xshift=12cm]
	 				\node[vertex,label=above:{$1$}] (r5) at (-0.45,1.5) {};
	 				\node[vertex,label=above right:{$2$}] (a5) at (0.25,0.75) {};
	 				\node[vertex,label=below right:{$\infty$}] (b5) at (0.95,0) {};
	 				
	 				\draw[edge] (r5)--(a5)--(b5);
	 				\node at (0.25,-0.65) {$T_5$};
	 			\end{scope}
	 			
	 		\end{tikzpicture}
	 			\caption{Five increasing binary trees of type $B$.}
	 		\label{fig:three-labeled-trees-B}
	 	\end{figure}
	 Our third result is concerned with the combinatorial interpretation of  $b_{n,k}(q)$  and $B_{n}(t,q)$ in terms of increasing binary trees of type $B$.
	 	 \begin{theorem}\label{thm-gamma-B}
	 		For   $n\geq 1$, we have 
	 		\begin{equation}\label{eq-gamma-2}
	 			B_n(t,q)=\sum\limits_{T\in \widetilde{\B}_n}t^{\omega(T)}q^{\fmaj(T)}=\sum\limits_{k=0}^{\lfloor {n/2}\rfloor}  
	 			b_{n, k}(q)t^{k}(-tq^{2k+1}; q^2)_{n-2k},
	 		\end{equation}
	 		where 
	 		$$
	 		b_{n,k}(q)=\sum_{T\in \widetilde{\B}^*_{n,k}}q^{\alpha_B(T)}.
	 		$$
	 	\end{theorem}
	 		 \begin{example}
	 		Take $n=2$ in Theorem~\ref{thm-gamma-B}. The expansion of $B_2(t,q)$ is given by
	 		\[
	 		B_2(t,q)=t(q^3+2q^2+q)+(1+tq)(1+tq^3)=t\sum_{i=1}^4q^{\alpha_B(T_i)}+ q^{\alpha_B(T_5)}(1+tq)(1+tq^3),
	 		\]
	 		where $T_1$, $T_2$,  $T_3$, $T_4$ and $T_5$ are shown in Figure~\ref{fig:three-labeled-trees-B} with $\alpha_B(T_1)=1$, $\alpha_B(T_2)=2$,   $\alpha_B(T_3)=2$, $\alpha_B(T_4)=3$ and $\alpha_B(T_5)=0$.
	 	\end{example}

	 	Let 
	 	$$
	 	\NT_{n,k}=\{T\in \widetilde{\B}_{n}\mid \leaf(T)=k+1,  \neg(T)=0\},
	 	$$
	 	where $\neg(T)$ denotes the number of negative   nodes of $T$. 
	 		In analogy to Theorem \ref{thm-d_{n,k}(q)},  we derive a combinatorial interpretation of the quotient $ {b_{n,k}(q)/(1+q)^k(-q; q^2)_k}$ via enumerative polynomials on certain   increasing binary trees of type $B$.
	 	\begin{theorem}\label{thm-e_{n,k}(q)}
	 		For   $n,k\geq 1$, we have
	 		$$
	 		b_{n,k}(q)=(1+q)^k(-q; q^2)_ke_{n,k}(q) ,
	 		$$
	 		where 
	 		$$
	 		e_{n,k}(q)= \sum_{T\in \NT_{n,k}}q^{\alpha_B(T)}.
	 		$$
	 		
	 	\end{theorem}

	 	In \cite{HanJouhetZeng2013qtriangles},  Han, Jouhet and Zeng introduced the following   $q$-analogue of $B_{2n}(-1) = (-1)^n 4^n E_{2n}$ (where the $E_{2n}$ are the famous secant numbers): 
	 	\[
	 	E_{2n}^{*}(q)=(-1)^nq^{n(n+1)}
	 	B_{2n}(-q^{-2n-1},q)
	 	\]
	 	They further  proved that 
	 	there exists a polynomial $G^*_{2n}(q)\in\mathbb Z[q]$ such that $G^*_{2n}(1)=E_{2n}$ and
	 	\[
	 	E^*_{2n}(q)=   (1+q)^n  (-q; q^2)_{n} G^*_{2n}(q)
	 	\]
	 	  and posed the following conjecture.
	 	  
	 	  \begin{conjecture}[\cite{HanJouhetZeng2013qtriangles}, Conjecture 8]\label{con}
	 	  	For    $n\geq 0$, all coefficients of the polynomials $G^*_{2n}(q)$ are positive.
	 	  	\end{conjecture}
	 Since $G^*_{2n}(1) = E_{2n}$, the polynomial $G^*_{2n}(q)$ can be viewed as a new   refinement of the
	 secant number.

	 		\begin{theorem}\label{thm:positive-G-star}
	 	For  $n\geq 0$, all coefficients of $G_{2n}^{*}(q)$ are positive
	 	integers. More precisely, for $n\geq 1$,
	 	\[
	 	G_{2n}^{*}(q)
	 	=\sum_{T\in\mathsf{NT}_{2n,n}}q^{\alpha_B(T)-n^2},
	 	\]
	 	and
	 	\[
	 	[q^j]G_{2n}^{*}(q)>0
	 	\qquad\text{for }0\leq j\leq n(n-1).
	 	\]
	 	In particular, the degree of $G^*_{2n}(q)$ is $n(n-1)$.
	 	 
	 	\end{theorem}
	 	
	 	The rest of this paper is organized as follows. In Section~3, we introduce a generalized Foata--Strehl action on increasing binary trees,  thereby proving the second equality stated in (\ref{eq-gamma-1}). Starting with the Carlitz insertion bijection on permutations, in Section~4 we construct a Carlitz‑type insertion bijection for increasing binary trees and establish a bijection between permutations and increasing binary trees. This proves the first equality in (\ref{eq-gamma-1}) and completes the proof of Theorem~\ref{thm-gamma}. In Section~5, we construct a generalized Foata--Strehl action  on increasing binary trees  of type $B$ and establish a combinatorial interpretation for the coefficients $b_{n,k}(q)$, thereby completing the proof of Theorem~\ref{thm-gamma-B}.   In Section~6,   we investigate   the combinatorial interpretations  of  the quotients 
	 	$a_{n,k}(q)/ (-q;q)_{k-1}$  and $b_{n,k}(q)/(1+q)^k(-q; q^2)_k$ and the polynomial $G^*_{2n}(q)$, thereby completing the proofs of Theorems \ref{thm-d_{n,k}(q)}, \ref{thm-e_{n,k}(q)} and  \ref{thm:positive-G-star}. 
	 	
	 		\section{A generalized Foata–Strehl action on  increasing binary trees}
	 	In this section, we develop a generalized Foata–Strehl action 
	 	on increasing binary  trees, thereby proving the second equality   stated in (\ref{eq-gamma-1}).
	 	
	 	By  simple computation,   it is plain to check that a binary tree $T\in \B_n$ satisfies the following relations. 
	 	\begin{lemma}\label{lem-tree}
	 		Given a binary tree $T\in \B_n$, we have
	 		\begin{equation}\label{eq-sint}
	 			\sint(T)=n+1-2\leaf(T),
	 		\end{equation}
	 		and 
	 		\begin{equation}\label{eq-left}
	 			\l(T)= \lint(T)+ \leaf(T)-1,
	 		\end{equation}
	 		where  $\sint(T)$  denotes the number of singleton internal nodes    of $T$.
	 	\end{lemma}
	 	
	 	Given an increasing binary  tree $T\in \B_n$  and a node $x$ of $T$,  we define an increasing binary tree $\theta_x(T)$ as follows:
	 	\begin{itemize}
	 		\item If node $x$ is a singleton internal node, let $\theta_x(T)$ be the tree obtained from $T$ by swapping the child type (left/right) of node $x$.
	 		\item Otherwise, set $\theta_x(T)=T$.
	 	\end{itemize}
	 	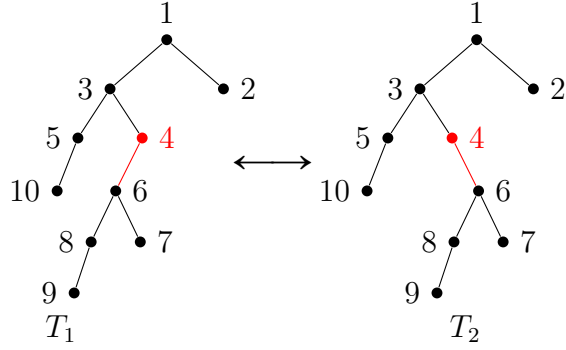
\begin{figure}[htbp]
	 		\centering
	 		
	 		\tikzset{
	 			vertex/.style={circle,fill=black,inner sep=1.45pt},
	 			active/.style={circle,fill=red,inner sep=1.45pt},
	 			edge/.style={thin},
	 			activeedge/.style={red,thin},
	 		}
	 		
	 		\begin{tikzpicture}[scale=0.5]
	 			
	 			
	 			\node[vertex,label=above:$1$] (a1) at (0,0) {};
	 			
	 			\node[vertex,label=left:$3$]  (a3) at (-1.5,-1.3) {};
	 			\node[vertex,label=right:$2$] (a2) at (1.5,-1.3) {};
	 			
	 			\node[vertex,label=left:$5$]  (a5) at (-2.35,-2.6) {};
	 			\node[active,label=right:\textcolor{red}{$4$}] (a4) at (-0.65,-2.6) {};
	 			
	 			\node[vertex,label=left:$10$] (a10) at (-2.9,-4.0) {};
	 			\node[vertex,label=right:$6$] (a6) at (-1.35,-4.0) {};
	 			
	 			\node[vertex,label=left:$8$]  (a8) at (-2.0,-5.35) {};
	 			\node[vertex,label=right:$7$] (a7) at (-0.7,-5.35) {};
	 			
	 			\node[vertex,label=left:$9$]  (a9) at (-2.45,-6.7) {};
	 			
	 			\draw[edge]
	 			(a1)--(a3)
	 			(a1)--(a2)
	 			(a3)--(a5)
	 			(a3)--(a4)
	 			(a5)--(a10)
	 			(a6)--(a8)
	 			(a6)--(a7)
	 			(a8)--(a9);
	 			
	 			\draw[activeedge] (a4)--(a6);
	 			
	 			\node at (-2.8,-7.65) {$T_1$};
	 			
	 			\node at (2.8,-3.35) {\Large $\longleftrightarrow$};

	 			
	 			\node[vertex,label=above:$1$] (b1) at (8.2,0) {};
	 			
	 			\node[vertex,label=left:$3$]  (b3) at (6.7,-1.3) {};
	 			\node[vertex,label=right:$2$] (b2) at (9.7,-1.3) {};
	 			
	 			\node[vertex,label=left:$5$]  (b5) at (5.85,-2.6) {};
	 			\node[active,label=right:\textcolor{red}{$4$}] (b4) at (7.55,-2.6) {};
	 			
	 			\node[vertex,label=left:$10$] (b10) at (5.3,-4.0) {};
	 			\node[vertex,label=right:$6$] (b6) at (8.25,-4.0) {};
	 			
	 			\node[vertex,label=left:$8$]  (b8) at (7.6,-5.35) {};
	 			\node[vertex,label=right:$7$] (b7) at (8.9,-5.35) {};
	 			
	 			\node[vertex,label=left:$9$]  (b9) at (7.15,-6.7) {};
	 			
	 			\draw[edge]
	 			(b1)--(b3)
	 			(b1)--(b2)
	 			(b3)--(b5)
	 			(b3)--(b4)
	 			(b5)--(b10)
	 			(b6)--(b8)
	 			(b6)--(b7)
	 			(b8)--(b9);
	 			
	 			\draw[activeedge] (b4)--(b6);
	 			
	 			\node at (7.9,-7.65) {$T_2$};
	 			
	 		\end{tikzpicture}
	 		
	 		\vspace{-2mm}
	 		\caption{An example of the transformation $\theta_{4}$.}
	 		\label{fig:theta4}
	 	\end{figure}
	 	
	 	For example, let  $T_1$ and $T_2$  be increasing binary trees  as shown in  Figure \ref{fig:theta4}.
	 	Then we have  $\theta_4(T_1)=T_2$ and $\theta_4(T_2)=T_1$.

	 	\begin{lemma}\label{lem-alpha}
	 		Let $T\in \B_n$. For any $x\in [n]$, we have 
	 		$\alpha(T)=\alpha(\theta_x(T))$. 
	 	\end{lemma}
	 \begin{proof}
	 	We proceed by induction on $n$. The statement is clear for $n=1$.
	 	Assume that $n\geq 2$ and that the statement holds for all trees
	 	in $\mathcal T_{n-1}$.
	 	
	 	Let $Q=\theta_x(T)$. If $x\notin\operatorname{Sint}(T)$, then
	 	$Q=T$, so assume that $x\in\operatorname{Sint}(T)$.
	 	Let $\widehat T$ and $\widehat Q$ be the trees obtained from
	 	$T$ and $Q$, respectively, by removing node $n$ together with
	 	its incident edge.
	 	
	 	If node $n$ is the child of node $x$, then
	 	$\widehat T=\widehat Q$, and $x$ is a leaf of this common tree.
	 	Suppose that $x$ is its $i$-th smallest leaf. By the first case
	 	in the definition of $\alpha$, we have
	 	\[
	 	\alpha(T)=\alpha(\widehat T)+i-1
	 	=\alpha(\widehat Q)+i-1=\alpha(Q).
	 	\]
	 	
	 	Otherwise, $x$ remains a singleton internal node in $\widehat T$,
	 	and $\widehat Q=\theta_x(\widehat T)$. In particular,
	 	\[
	 	\operatorname{Leaf}(\widehat T)
	 	=\operatorname{Leaf}(\widehat Q)
	 	\quad\text{and}\quad
	 	\operatorname{Sint}(\widehat T)
	 	=\operatorname{Sint}(\widehat Q).
	 	\]
	 	Moreover, node $n$ has the same parent and the same left/right
	 	position in $T$ and $Q$. Hence the recursive definition of
	 	$\alpha$ gives
	 	\[
	 	\alpha(T)-\alpha(\widehat T)
	 	=
	 	\alpha(Q)-\alpha(\widehat Q).
	 	\]
	 	By the induction hypothesis,
	 	$\alpha(\widehat T)=\alpha(\widehat Q)$, and therefore
	 	$\alpha(T)=\alpha(Q)$, completing the proof.
	 \end{proof}
	 By the definition of $\theta_x$, it is plain to see that the number of leaves is invariant under the transformation $\theta_x$. This, together with the definition of $\theta_x(T)$ and Lemma \ref{lem-alpha}, ensures that the transformation $\theta_x$ has the following desired properties.
	 	
	 	\begin{lemma}\label{lem-theta}
	 		Let $T\in \B_n$. For any $x\in [n]$, we have $\theta_x(T)\in \B_n$. Moreover,  the resulting tree $\theta_x(T)$ verifies the following properties.
	 		\begin{itemize}
	 			\item The node $x$ is a left (resp.,  right)   internal node of $T$  if and only 
	 			if  the node $x$ is a right  (resp.,  left)   internal node of $\theta_x(T)$.
	 			\item  If node $x$ is a left   internal node of $T$ and  is the $y$-th smallest singleton internal node of $T$, then we have 
	 			$$
	 			\maj(\theta_x(T))=\maj(T)-\leaf(T)-y+1,
	 			$$
	 			and 
	 			$$
	 			\l(\theta_x(T))=\l(T)-1.
	 			$$
	 			
	 			\item  If    node $x$ is a right    internal node of $T$ and is the $y$-th smallest singleton internal node of $T$, then we have 
	 			$$
	 			\maj(\theta_x(T))=\maj(T)+\leaf(T)+y-1,
	 			$$
	 			and 
	 			$$
	 			\l(\theta_x(T))=\l(T)+1.
	 			$$
	 		\end{itemize}
	 	\end{lemma}

	 	In view of Lemma \ref{lem-theta}, it is not hard to see that the transformation $\theta_x$ is an involution
	 	on $\B_n$ and $\theta_x$ and $\theta_y$ commute for all $x,y\in [n]$. For any $S\subseteq [n]$, we can define the
	 	function $\theta_S\colon \B_n\to \B_n$ by
	 	\[
	 	\theta_S=\prod_{x\in S}\theta_x,
	 	\]
	 	where the product is the functional composition. Hence the group $\mathbb{Z}_2^{n}$
	 	acts on $\B_n$ via the functions $\theta_S$ for $S\subseteq [n]$.
	 	Such an action is called the {\em generalized Foata–Strehl action }(GFS‑action for short) on increasing binary trees.
	 	
	 	 \begin{theorem}\label{thm-expansion}
	 	 	For $n\geq 1$,  we have
	 	 		\begin{equation}\label{eq-expansion}
	 	 		\sum\limits_{T\in \B_n}t^{\l(T)}q^{\maj(T)}=\sum\limits_{k=1}^{\lfloor {n+1\over 2}\rfloor}  
	 	 		a_{n, k}(q)t^{k-1}\prod\limits_{j=k}^{n-k} (1+tq^j),
	 	 	\end{equation}
	 	 	where 
	 	 	$$
	 	 	a_{n,k}(q)=\sum_{T\in \B^*_{n,k}}q^{\alpha(T)}.
	 	 	$$
	 	 	\end{theorem}
	 	\begin{proof}
	 			For each $T\in\B_n$, we define $[T]$ to be the set of trees $S$ that can be transformed
	 		to $T$ via the GFS-action on  increasing binary trees, which is called the  orbit of $T$.  Clearly, the GFS-action
	 		divides $\B_n$ into disjoint orbits. In view of Lemma~\ref{lem-theta}, each $[T]$ contains exactly one
	 		tree without any left  internal nodes that we denote by $\bar{T}$.  By (\ref{eq-sint}) and (\ref{eq-left}), we have
	 		\begin{equation}\label{eq-bar-T-1}
	 			\sint(\bar{T})=n-2\leaf(\bar{T})+1 
	 		\end{equation}
	 		and 
	 		\begin{equation}\label{eq-bar-T-2}
	 			\leaf(\bar{T})=\l(\bar{T})+1.
	 		\end{equation}
	 		
	 		Then, by Lemma~\ref{lem-theta}, we have
	 		\[
	 		\begin{array}{lll}
	 			\sum\limits_{T\in [\bar{T}]} t^{\l(T)}q^{\maj(T)}
	 			&= & t^{\l(\bar{T})}q^{\maj(\bar{T})}\prod\limits_{j=1}^{n-2\leaf(\bar{T})+1}(1+tq^{\leaf(\bar{T})+j-1})\\
	 			&&\\
	 			&=  &t^{\leaf(\bar{T})-1}q^{\maj(\bar{T})}\prod\limits_{j=\leaf(\bar{T})}^{n-\leaf(\bar{T})}(1+tq^j)\\
	 			&&\\
	 			&=  &t^{\leaf(\bar{T})-1}q^{\alpha(\bar{T})}\prod\limits_{j=\leaf(\bar{T})}^{n-\leaf(\bar{T})}(1+tq^j)
	 		\end{array}
	 		\]
	 		where the first equality follows from (\ref{eq-bar-T-1}), the second equality follows from (\ref{eq-bar-T-2}) and the last equality follows from the relation $\maj(\bar{T})=\alpha(\bar{T})$.  Recall that $\B^*_{n,k}$ denotes the set of increasing binary trees without any left  internal nodes and with exactly $k$ leaves.
	 		Then, we derive that
	 		\[
	 		\begin{aligned}
	 			\sum\limits_{T\in \B_n}t^{\l(T)}q^{\maj(T)}
	 			&= \sum_{k\geq 1}\sum_{\bar{T}\in\B^*_{n,k} }\sum_{T\in [\bar{T}]} t^{\l(T)}q^{\maj(T)}\\
	 			&= \sum_{k\geq 1}\sum_{\bar{T}\in\B^*_{n,k} }t^{k-1}q^{\alpha(\bar{T})}\prod\limits_{j=k}^{n-k}(1+tq^j)\\
	 			&=  \sum_{k\geq 1} a_{n,k}(q)t^{k-1}\prod\limits_{j=k}^{n-k}(1+tq^j),
	 		\end{aligned}
	 		\]
	 		completing the proof. 
	 		
	 		\end{proof}

	 	\section{Bijection between permutations and increasing binary trees}
	 	In this section, we establish  a bijection between   permutations and increasing
	 	binary trees, thereby proving the first equality  stated in (\ref{eq-gamma-1}).
	 	
	 	 \begin{theorem}\label{thm-Phi}
	 	 	For $n\geq 1$, there exists a bijection $\Phi: \mathfrak{S}_n\rightarrow \B_n$ such that for any permutation $\pi\in \mathfrak{S}_n$, we have
	 	 	$$
	 	 	(\des, \maj)\pi= (\l, \maj)\Phi(\pi).
	 	 	$$
	 	 	\end{theorem}
	 	
	 	\subsection{The Carlitz's  insertion  bijection  for permutations}
	 	 In this subsection, we give a review of Carlitz's insertion bijection \cite{Carlitz}.   Given $\pi\in \mathfrak{S}_{n-1}$, let 
	 	$$
	 	S(\pi)=\{\pi_{i}\mid \pi_{i-1}>\pi_i, \,\, 2\leq i\leq n-1\} 
	 	$$
	 	Clearly, we have $|S(\pi)|=\des(\pi)$. The {\em $\maj$-labeling} of $\pi$  is obtained as follows.
	 	\begin{itemize}
	 		\item Label the  space after $\pi_{n-1}$ by $0$.
	 		\item Label the spaces before those letters of $\pi$ that are belonging to $S(\pi)$ from right to left with $1,2,\ldots, \des(\pi)$. 
	 		\item Label the remaining spaces from left to right with $\des(\pi)+1,  \ldots, n-1$.
	 	\end{itemize}

	 	For example, if we   let    $\pi=473598216$,  then we have   $S(\pi)=\{1,2,3, 8 \}$,  and thus  the $\maj$-labeling of $\pi$ is given by
	 	$$
	 	\begin{array} {llllllllll}
	 		_{ 5}4&_67&_{\red 4} 3&_75&_89&_{\red3}8&_{\red2}2&_{\red 1}1&_96&_0,
	 	\end{array}
	 	$$
	 	where the labels of the spaces are written as subscripts. 
	 	Define $$\phi:\mathfrak{S}_{n-1}\times  \{0, 1, \ldots, n-1\} \longrightarrow \mathfrak{S}_{n}$$ by mapping the pair $( \pi, c)$ to the  permutation in $\mathfrak{S}_{n}$  obtained by inserting   $ n$ at the space in $\pi$ which  is labeled by $c$ in the $\mathrm{maj}$-labeling of $\pi$. Continuing with our running example, we have
	 	$\phi(473598216, 8)= 4735(10)98216$.

	 	Carlitz \cite{Carlitz}  proved that the map $\phi$  verifies  the following celebrated  properties.
	 		\begin{theorem}[Carlitz~ \cite{Carlitz}] \label{thm-Car}
	 	Fix $n\geq 2$. 	The map $\phi\colon \mathfrak{S}_{n-1}\times \{0,1,\dots,n-1\}\rightarrow  \mathfrak{S}_{n}$ is a bijection such that
	 		for  any permutation $\pi\in \mathfrak{S}_{n-1}$ and $0\le c\le n-1$, we have
	 		\begin{equation}\label{eq-phi-1}
	 		\des(\phi(\pi, c))=\left\{ 
	 		 \begin{array}{ll}
	 			\des(\pi) & \text{if } 0\le c\le \des(\pi),\\[2pt]
	 			\des(\pi)+1 & \text{otherwise},
	 	\end{array}
	 	\right.
	 		\end{equation}
	 		and
	 		\begin{equation}\label{eq-phi-2}
	 		\maj(\phi(\pi, c))=\maj(\pi)+c.
	 		\end{equation}
	 	\end{theorem}

	 \subsection{The Carlitz-type insertion bijection for trees }
	 In this subsection, we will develop a Carlitz-type insertion bijection for  increasing binary trees which maps a  pair $(T,c)$ with $T\in \B_{n-1}$ and $c\in \{0,1,\ldots, n-1\}$ to an increasing binary tree in $\B_n$.

	 \noindent{\bf The map $\psi: \B_{n-1}\times \{0,1, \ldots, n-1\}\rightarrow \B_n$.}
	 
	 	Given $T\in \B_{n-1}$, let 
	 	$$
	 	\Leaf(T)=\{\ell_1<\ell_2<\cdots<\ell_k\}
	 	$$
	  and
	 	$$\Sint(T)=\{u_1<u_2<\cdots< u_{n-2k}\}.$$  Assume that $T$ has exactly $s$ left  internal nodes. Let 
	 	$$
	 	\Lint(T)=\{u_{i_1}<u_{i_2}<\cdots<u_{i_s}\} 
	 	$$
	 	and 
	 	$$
	 	\Rint(T)=\{u_{j_1}>u_{j_2}>\cdots>u_{j_t}\}
	 	$$
	 	denote the set of  left  internal nodes  of $T$ and the set of right   internal nodes of $T$, respectively.   Clearly, we have $t=n-2k-s$.

	 	 Let $\beta(T)=(\beta_1, \beta_2, \ldots, \beta_{n-2k})$.
	 Given $c\in \{0,1,\dots,n-1\}$, define $Q=\psi(T,c)$ as the tree constructed according to the following four cases.
	 
	 	\noindent \textbf{Case $\mathrm{(i)}$:} $0\le c\le k-1$.\\
	 	Connect node $\ell_{c+1}$ and node $n$ by a right edge, and denote the resulting tree by $P$.
	 	Construct $Q$ from $P$ by adjusting   the child-type (left/right)   of each singleton internal node of $P$ so that $\beta(Q)=(\beta_1, \beta_2, \ldots, \beta_{n-2k}, 0)$.  See Figure \ref{fig:case1-psi} for an illustration,  where $n=7$, $k=2$, $\Leaf(T)=\{\ell_1, \ell_2\}=\{2,6\}$ and $c=0$.

	 	\noindent \textbf{Case $\mathrm{(ii)}$:} $n-k\le c\le n-1$.\\
	 	Let $r=c-n+k+1$.
	 	Connect node $\ell_{r}$ and node $n$ by a left edge, and denote the resulting tree by $P$.
	 Construct $Q$ from $P$ by adjusting the child-type (left/right)   of each singleton internal node of $P$ so that $\beta(Q)=(\beta_1, \beta_2, \ldots, \beta_{n-2k}, 1)$. See Figure \ref{fig:case2-psi} for an illustration,  where $n=7$, $k=2$,  $\Leaf(T)=\{\ell_1, \ell_2\}=\{2,6\}$, $c=5$ and $r=1$.

	 	\noindent \textbf{Case $\mathrm{(iii)}$:} $k\le c\le k-1+s$.\\
	 	Let $r=c-k+1$. Define $Q$ to be the tree obtained from $T$ by connecting node $u_{i_{r}}$ and node $n$ via a right edge.  See Figure \ref{fig:case3-psi} for an illustration,  where $n=7$, $k=2$, $s=1$,  $\Lint(T)=\{ 3\}$, $c=2$ and $r=1$. 
	 	
	 	\noindent \textbf{Case $\mathrm{(iv)}$:} $k+s\le c\le n-k-1$.\\
	 	Let $r=c-k-s+1$. Define $Q$ to be the tree obtained from $T$ by connecting node $u_{j_{r}}$ and node $n$ via a left edge.  See Figure \ref{fig:case4-psi} for an illustration,  where $n=7$, $k=2$, $s=1$,  $\Rint(T)=\{ u_{j_1}, u_{j_2}\}=\{5,4\}$, $c=3$ and $r=1$.

	 		\tikzset{
	 		v/.style={circle,fill,inner sep=1.05pt},
	 		arr/.style={->,>=stealth,thin}
	 	}
	 	
	 		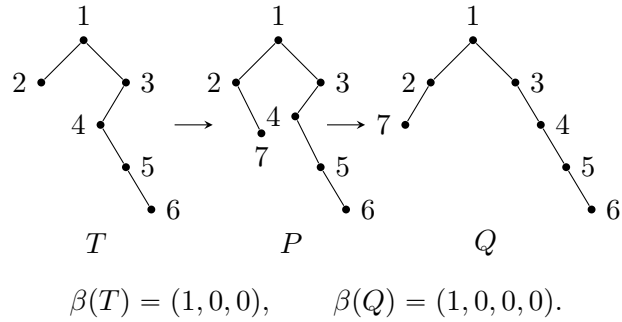
\begin{figure}[htbp]
	 		\centering
	 		\small
	 		\vspace{-2mm}

	 		\begin{tikzpicture}[scale=.56]
	 			
	 			\node[v,label=above:$1$] (a1) at (0,0) {};
	 			\node[v,label=left:$2$]  (a2) at (-1,-1) {};
	 			\node[v,label=right:$3$] (a3) at (1,-1) {};
	 			\node[v,label=left:$4$]  (a4) at (.4,-2) {};
	 			\node[v,label=right:$5$] (a5) at (1,-3) {};
	 			\node[v,label=right:$6$] (a6) at (1.6,-4) {};
	 			\draw (a1)--(a2) (a1)--(a3)
	 			(a3)--(a4) (a4)--(a5) (a5)--(a6);
	 			\node at (.3,-4.8) {$T$};
	 			
	 			\draw[arr] (2.15,-2) -- (3.05,-2);
	 			
	 			\node[v,label=above:$1$] (b1) at (4.6,0) {};
	 			\node[v,label=left:$2$]  (b2) at (3.6,-1) {};
	 			\node[v,label=right:$3$] (b3) at (5.6,-1) {};
	 			\node[v,label=below:$7$] (b7) at (4.2,-2.2) {};
	 			\node[v,label=left:$4$]  (b4) at (5,-1.8) {};
	 			\node[v,label=right:$5$] (b5) at (5.6,-3) {};
	 			\node[v,label=right:$6$] (b6) at (6.2,-4) {};
	 			\draw (b1)--(b2) (b1)--(b3)
	 			(b2)--(b7) (b3)--(b4)
	 			(b4)--(b5) (b5)--(b6);
	 			\node at (4.9,-4.8) {$P$};
	 			
	 			\draw[arr] (5.75,-2) -- (6.65,-2);
	 			
	 			\node[v,label=above:$1$] (c1) at (9.2,0) {};
	 			\node[v,label=left:$2$]  (c2) at (8.2,-1) {};
	 			\node[v,label=right:$3$] (c3) at (10.2,-1) {};
	 			\node[v,label=left:$7$]  (c7) at (7.6,-2) {};
	 			\node[v,label=right:$4$] (c4) at (10.8,-2) {};
	 			\node[v,label=right:$5$] (c5) at (11.4,-3) {};
	 			\node[v,label=right:$6$] (c6) at (12,-4) {};
	 			\draw (c1)--(c2) (c1)--(c3)
	 			(c2)--(c7) (c3)--(c4)
	 			(c4)--(c5) (c5)--(c6);
	 			\node at (9.5,-4.8) {$Q$};
	 			
	 		\end{tikzpicture}
	 		
	 		\smallskip
	 		$\beta(T)=(1,0,0),\qquad \beta(Q)=(1,0,0,0).$
	 		
	 		\vspace{-1mm}
	 		\caption{An example of Case (i) of the map $\psi$.}
	 		\label{fig:case1-psi}
	 		\vspace{-2mm}
	 	\end{figure}

	 	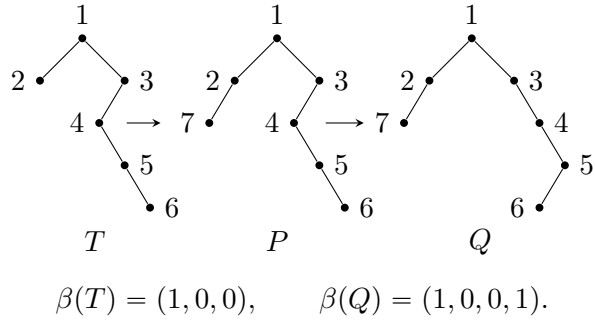
\begin{figure}[htbp]
	 		\centering
	 		\small
	 		\vspace{-2mm}
	 		 
	 		\begin{tikzpicture}[scale=.56]
	 			
	 			\node[v,label=above:$1$] (a1) at (0,0) {};
	 			\node[v,label=left:$2$]  (a2) at (-1,-1) {};
	 			\node[v,label=right:$3$] (a3) at (1,-1) {};
	 			\node[v,label=left:$4$]  (a4) at (.4,-2) {};
	 			\node[v,label=right:$5$] (a5) at (1,-3) {};
	 			\node[v,label=right:$6$] (a6) at (1.6,-4) {};
	 			\draw (a1)--(a2) (a1)--(a3)
	 			(a3)--(a4) (a4)--(a5) (a5)--(a6);
	 			\node at (.3,-4.8) {$T$};
	 			
	 			\draw[arr] (1.05,-2) -- (1.85,-2);
	 			
	 			\node[v,label=above:$1$] (b1) at (4.6,0) {};
	 			\node[v,label=left:$2$]  (b2) at (3.6,-1) {};
	 			\node[v,label=right:$3$] (b3) at (5.6,-1) {};
	 			\node[v,label=left:$7$]  (b7) at (3,-2) {};
	 			\node[v,label=left:$4$]  (b4) at (5,-2) {};
	 			\node[v,label=right:$5$] (b5) at (5.6,-3) {};
	 			\node[v,label=right:$6$] (b6) at (6.2,-4) {};
	 			\draw (b1)--(b2) (b1)--(b3)
	 			(b2)--(b7) (b3)--(b4)
	 			(b4)--(b5) (b5)--(b6);
	 			\node at (4.6,-4.8) {$P$};
	 			
	 			\draw[arr] (5.75,-2) -- (6.65,-2);
	 			
	 			\node[v,label=above:$1$] (c1) at (9.2,0) {};
	 			\node[v,label=left:$2$]  (c2) at (8.2,-1) {};
	 			\node[v,label=right:$3$] (c3) at (10.2,-1) {};
	 			\node[v,label=left:$7$]  (c7) at (7.6,-2) {};
	 			\node[v,label=right:$4$] (c4) at (10.8,-2) {};
	 			\node[v,label=right:$5$] (c5) at (11.4,-3) {};
	 			\node[v,label=left:$6$]  (c6) at (10.8,-4) {};
	 			\draw (c1)--(c2) (c1)--(c3)
	 			(c2)--(c7) (c3)--(c4)
	 			(c4)--(c5) (c5)--(c6);
	 			\node at (9.4,-4.8) {$Q$};
	 			
	 		\end{tikzpicture}
	 		
	 		\smallskip
	 		$\beta(T)=(1,0,0),\qquad \beta(Q)=(1,0,0,1).$
	 		
	 		\vspace{-1mm}
	 		\caption{An example of Case (ii) of the map $\psi$.}
	 		\label{fig:case2-psi}
	 		\vspace{-2mm}
	 	\end{figure}

	 	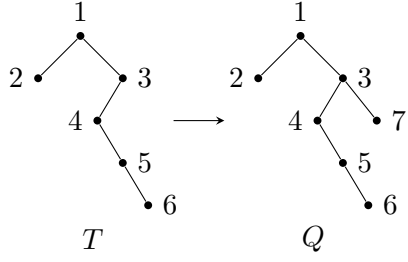
\begin{figure}[htbp]
	 		\centering
	 		\small
	 		\vspace{-2mm}

	 		\begin{tikzpicture}[scale=.56]
	 			
	 			\node[v,label=above:$1$] (a1) at (0,0) {};
	 			\node[v,label=left:$2$]  (a2) at (-1,-1) {};
	 			\node[v,label=right:$3$] (a3) at (1,-1) {};
	 			\node[v,label=left:$4$]  (a4) at (.4,-2) {};
	 			\node[v,label=right:$5$] (a5) at (1,-3) {};
	 			\node[v,label=right:$6$] (a6) at (1.6,-4) {};
	 			\draw (a1)--(a2) (a1)--(a3)
	 			(a3)--(a4) (a4)--(a5) (a5)--(a6);
	 			\node at (.3,-4.8) {$T$};
	 			
	 			\draw[arr] (2.2,-2) -- (3.4,-2);
	 			
	 			\node[v,label=above:$1$] (b1) at (5.2,0) {};
	 			\node[v,label=left:$2$]  (b2) at (4.2,-1) {};
	 			\node[v,label=right:$3$] (b3) at (6.2,-1) {};
	 			\node[v,label=left:$4$]  (b4) at (5.6,-2) {};
	 			\node[v,label=right:$7$] (b7) at (7,-2) {};
	 			\node[v,label=right:$5$] (b5) at (6.2,-3) {};
	 			\node[v,label=right:$6$] (b6) at (6.8,-4) {};
	 			\draw (b1)--(b2) (b1)--(b3)
	 			(b3)--(b4) (b3)--(b7)
	 			(b4)--(b5) (b5)--(b6);
	 			\node at (5.5,-4.8) {$Q$};
	 			
	 		\end{tikzpicture}

	 		\vspace{-1mm}
	 		\caption{An example of Case (iii)  of the map $\psi$.}
	 		\label{fig:case3-psi}
	 		\vspace{-2mm}
	 	\end{figure}

	 	\begin{figure}[htbp]
	 		\centering
	 		\small
	 		\vspace{-2mm}

	 		\begin{tikzpicture}[scale=.56]
	 			
	 			\node[v,label=above:$1$] (a1) at (0,0) {};
	 			\node[v,label=left:$2$]  (a2) at (-1,-1) {};
	 			\node[v,label=right:$3$] (a3) at (1,-1) {};
	 			\node[v,label=left:$4$]  (a4) at (.4,-2) {};
	 			\node[v,label=right:$5$] (a5) at (1,-3) {};
	 			\node[v,label=right:$6$] (a6) at (1.6,-4) {};
	 			\draw (a1)--(a2) (a1)--(a3)
	 			(a3)--(a4) (a4)--(a5) (a5)--(a6);
	 			\node at (.3,-4.8) {$T$};
	 			
	 			\draw[arr] (2.2,-2) -- (3.4,-2);
	 			
	 			\node[v,label=above:$1$] (b1) at (5.2,0) {};
	 			\node[v,label=left:$2$]  (b2) at (4.2,-1) {};
	 			\node[v,label=right:$3$] (b3) at (6.2,-1) {};
	 			\node[v,label=left:$4$]  (b4) at (5.6,-2) {};
	 			\node[v,label=right:$5$] (b5) at (6.2,-3) {};
	 			\node[v,label=left:$7$]  (b7) at (5.5,-4) {};
	 			\node[v,label=right:$6$] (b6) at (6.9,-4) {};
	 			\draw (b1)--(b2) (b1)--(b3)
	 			(b3)--(b4) (b4)--(b5)
	 			(b5)--(b7) (b5)--(b6);
	 			\node at (5.5,-4.8) {$Q$};
	 			
	 		\end{tikzpicture}

	 		\vspace{-1mm}
	 		\caption{An example of Case (iv)  of the map  $\psi$.}
	 		\label{fig:case4-psi}
	 		\vspace{-2mm}
	 	\end{figure}
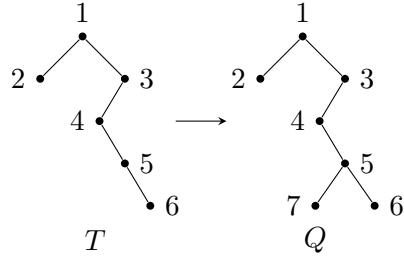
	 	
	 	\begin{lemma}\label{lem-psi-1}
	 		
	 	Fix $n\geq 2$. 	For  any tree $T\in \B_{n-1}$ and $0\le c\le n-1$, we have
	 		\[
	 		\l(\psi(T, c))=\left\{ 
	 		\begin{array}{ll}
	 			\l(T) & \mbox{if  $0\leq c\leq \l(T)$},\\[2pt]
	 			\l(T)+1 & \mbox{otherwise},
	 		\end{array}
	 		\right.
	 		\]
	 		and
	 		\[
	 		\maj(\psi(T, c))=\maj(T)+c.
	 		\]
	 		\end{lemma}
	 		\begin{proof}
	 			Here we retain the notations in the definition of the map $\psi$. Let $Q=\psi(T,c)$.   We proceed  by considering  four cases.
	 			
	 			\noindent{\bf Case 1: $0\leq c\leq k-1.$ }\\
	 			By Lemma \ref{lem-tree}, we have $\l(T)=\lint(T)+\leaf(T)-1$.  Then $c\leq k-1=\leaf(T)-1$ would imply that $c\leq \l(T)$.  Now we proceed to show that 
	 			$$
	 			\maj(Q)=\maj(T)+c \,\, \mbox{and}\,\, \l(Q)=\l(T).
	 			$$
	 		 
	 			Recall that $P$ is obtained from $T$ by connecting the node $\ell_{c+1}$ and node $n$ by a right edge.  It is plain to see that $$\l(T)=\l(P),$$ $$\leaf(T)=\leaf(P),$$   and $$\alpha(P)=\alpha(T)+c.$$  
	 			Recall that $Q$ is obtained from $P$ by adjusting the child-type (left/right)   of each singleton internal node of $P$.  Moreover, we have  $\beta(Q)=(\beta_1, \beta_2, \ldots, \beta_{n-2k}, 0)$ and $\beta(T)=(\beta_1, \beta_2, \ldots, \beta_{n-2k})$.  By Lemma \ref{lem-alpha},   adjusting the child-type (left/right)   of each singleton internal node of $P$ in the procedure from $P$ to $Q$ does not affect the statistic $\alpha(P)$.   It is easily seen that the procedure from $P$ to $Q$ preserves the number of leaves and the number of left edges. This gives that  $$\alpha(Q)=\alpha(P)=\alpha(T)+c,$$
	 			 $$\leaf(Q)=\leaf(P)=\leaf(T),$$
	 			 and $$\l(Q)=\l(P)=\l(T).
	 			 $$
	 			  Hence, we deduce that
	 			 $$
	 			 \begin{array}{lll}
	 			 \maj(Q)&=&\alpha(Q)+\sum\limits_{i=1}^{n-2k+1}(\leaf(Q)+i-1)\beta_i\\
	 			 &=&\alpha(T)+c+\sum\limits_{i=1}^{n-2k}(\leaf(T)+i-1)\beta_i\\
	 			 &=& \maj(T)+c,
	 			 \end{array}
	 			 $$
	 			as desired.

	 		 	\noindent{\bf Case 2: $n-k\leq c\leq n-1.$ }\\
	 		 By Lemma \ref{lem-tree}, we have $\l(T)=\lint(T)+\leaf(T)-1\leq n-2k+k-1=n-k-1$.  Then $c\geq n-k$ would imply that $c> \l(T)$.  Now we proceed to show that 
	 		 $$
	 		 \maj(Q)=\maj(T)+c \,\, \mbox{and}\,\, \l(Q)=\l(T)+1.
	 		 $$
	 		 
	 		 Recall that $P$ is obtained from $T$ by connecting the node $\ell_r$ and node $n$ by a left edge where $r=c-n+k+1$.  It is plain to see that $$\l(P)=\l(T)+1,$$ $$\leaf(T)=\leaf(P),$$  and  $$\alpha(P)=\alpha(T)+r-1.$$  
	 		 Recall that $Q$ is obtained from $P$ by adjusting the child-type (left/right)   of each singleton internal node of $P$.  Moreover, we have  $\beta(Q)=(\beta_1, \beta_2, \ldots, \beta_{n-2k}, 1)$ and $\beta(T)=(\beta_1, \beta_2, \ldots, \beta_{n-2k})$.  By Lemma \ref{lem-alpha},   adjusting the child-type (left/right)   of each singleton internal node of $P$ in the procedure from $P$ to $Q$ does not affect the statistic $\alpha(P)$.   It is easily seen that the procedure from $P$ to $Q$ preserves the number of leaves and the number of left edges. This gives that  $$\alpha(Q)=\alpha(P)=\alpha(T)+r-1,$$
	 		 $$\leaf(Q)=\leaf(P)=\leaf(T)=k,$$
	 		 and $$\l(Q)=\l(P)=\l(T)+1.
	 		 $$
	 		 Hence, we deduce that
	 		 $$
	 		 \begin{array}{lll}
	 		 	\maj(Q)&=&\alpha(Q)+(\leaf(Q)+n-2k)+\sum\limits_{i=1}^{n-2k}(\leaf(Q)+i-1)\beta_i\\
	 		 	&=&\alpha(T)+r-1+k+n-2k+\sum\limits_{i=1}^{n-2k}(\leaf(T)+i-1)\beta_i\\
	 		 	&=& \maj(T)+c
	 		 \end{array}
	 		 $$
	 		 as desired.
	 		 
	 		 	\noindent{\bf Case 3: $k\leq c\leq k+s-1.$ }\\
	 		 By Lemma \ref{lem-tree}, we have $\l(T)=\lint(T)+\leaf(T)-1=k+s-1$.  Then $c\leq k+s-1$ would imply that $c\leq  \l(T)$.  Now we proceed to show that 
	 		 $$
	 		 \maj(Q)=\maj(T)+c \,\, \mbox{and}\,\, \l(Q)=\l(T).
	 		 $$
	 		 Recall that $Q$ is obtained from $T$ by connecting the node $u_{i_r}$ and node $n$ by a right edge where $r=c-k+1$.  It is plain to see that $$\l(Q)=\l(T),$$ 
	 		 $$\leaf(Q)=\leaf(T)+1$$ and   $$\alpha(Q)=\alpha(T)+2\leaf(T)+i_r-1.$$    Moreover, we have  $$\beta(Q)=(\beta_1, \beta_2, \ldots, \beta_{i_r-1},   \beta_{i_r+1}, \ldots,  \beta_{n-2k})$$ and $$\beta(T)=(\beta_1, \beta_2, \ldots, \beta_{i_r-1}, 1,  \beta_{i_r+1}, \ldots,  \beta_{n-2k}).$$  
	 		 Hence, we deduce that
	 		 $$
	 		 \begin{array}{lll}
	 		 	\maj(Q)&=&\alpha(Q)+\sum\limits_{j=1}^{i_r-1}(\leaf(Q)+j-1)\beta_j+\sum\limits_{j=i_r+1}^{n-2k}(\leaf(Q)+j-2)\beta_j\\
	 		 	&=&\alpha(T) +2\leaf(T)+i_r-1+\sum\limits_{j=1}^{i_r-1}(\leaf(T)+j)\beta_j+\sum\limits_{j=i_r+1}^{n-2k}(\leaf(T)-1+j)\beta_j\\
	 		 	&=&  \alpha(T) +\leaf(T)+\sum\limits_{j=1}^{i_r-1}\beta_j+\sum\limits_{j=1}^{n-2k}(\leaf(T)+j-1)\beta_j\\
	 		  
	 		 	&=&  \alpha(T) +k+r-1+\sum\limits_{j=1}^{n-2k}(\leaf(T)+j-1)\beta_j\\
	 		 	&=& \maj(T)+c 
	 		 \end{array}
	 		 $$
	 		 as desired.
	 		 
	 			\noindent{\bf Case 4: $k+s\leq c\leq n-k-1.$ }\\
	 		By Lemma \ref{lem-tree}, we have $\l(T)=\lint(T)+\leaf(T)-1=k+s-1$.  Then $c\geq k+s$ would imply that $c>  \l(T)$.  Now we proceed to show that 
	 		$$
	 		\maj(Q)=\maj(T)+c \,\, \mbox{and}\,\, \l(Q)=\l(T)+1.
	 		$$

	 		Recall that $Q$ is obtained from $T$ by connecting the node $u_{j_r}$ and node $n$ by a left edge where $r=c-k-s+1$.   Clearly, we have
	 		\begin{equation}\label{eq-proof-1}
	 			\# \{u_i\mid i<j_r, u_i\in \Lint(T)\}+\# \{u_i\mid i>j_r, u_i\in \Lint(T)\}=s, 
	 		\end{equation}
	 		\begin{equation}\label{eq-proof-2}
	 			\# \{u_i\mid i>j_r, u_i\in \Lint(T)\}+r= \# \{u_i\mid i\geq j_r, u_i\in \Sint(T)\}=n-2k-j_r+1,
	 		\end{equation}
	 		and 
	 		\begin{equation}\label{eq-proof-3}
	 			\# \{u_i\mid i<j_r, u_i\in \Lint(T)\}=\sum\limits_{i=1}^{j_r-1}\beta_i.
	 		\end{equation}
	 		Invoking (\ref{eq-proof-1})-(\ref{eq-proof-3}), we deduce that
	 		\begin{equation}\label{eq-proof-4}
	 			n-2k-j_r+\sum\limits_{i=1}^{j_r-1}\beta_i=r+s-1.
	 		\end{equation}
	 		
	 		  It is plain to see that $$\l(Q)=\l(T)+1,$$ 
	 		$$\leaf(Q)=\leaf(T)+1$$ and   $$\alpha(Q)=\alpha(T)+\leaf(T)+n-2k-j_r.$$    Moreover, we have  $$\beta(Q)=(\beta_1, \beta_2, \ldots, \beta_{j_r-1},   \beta_{j_r+1},  \ldots, \beta_{n-2k})$$ and $$\beta(T)=(\beta_1, \beta_2, \ldots, \beta_{j_r-1}, 0,  \beta_{j_r+1}, \ldots,  \beta_{n-2k}).$$  
	 		Hence, we deduce that
	 		$$
	 		\begin{array}{lll}
	 			\maj(Q)&=&\alpha(Q)+\sum\limits_{i=1}^{j_r-1}(\leaf(Q)+i-1)\beta_i+\sum\limits_{i=j_r+1}^{n-2k}(\leaf(Q)+i-2)\beta_i\\
	 			&=&\alpha(T)+\leaf(T) +n-2k-j_r+\sum\limits_{i=1}^{j_r-1}(\leaf(T)+i)\beta_i+\sum\limits_{i=j_r+1}^{n-2k}(\leaf(T)-1+i)\beta_i\\
	 			&=&  \alpha(T)+\leaf(T) +n-2k-j_r+\sum\limits_{i=1}^{j_r-1}\beta_i+\sum\limits_{i=1}^{n-2k}(\leaf(T)+i-1)\beta_i\\
	 			
	 			&=&  \alpha(T) +k+r+s-1+\sum\limits_{i=1}^{n-2k}(\leaf(T)+i-1)\beta_i\\
	 			&=& \maj(T)+c 
	 		\end{array}
	 		$$
	 		as desired, where the fourth equality follows from \eqref{eq-proof-4}.  This completes the proof.

	 			\end{proof}
	 		\begin{theorem}\label{thm-psi}
	 			Fix $n\geq 2$.  The map $\psi\colon \B_{n-1}\times \{0,1,\dots,n-1\}\rightarrow  \B_{n}$ is a bijection such that for any tree  $T\in \B_{n-1}$ and $0\le c\le n-1$, we have
	 	\begin{gather}
	 		\l(\psi(T, c))=
	 		\begin{cases}
	 			\l(T) & \text{if } 0\le c\le \l(T),\\[2pt]
	 			\l(T)+1 & \text{otherwise},
	 		\end{cases}
	 		\label{eq-psi-1}\\[4pt]
	 		\maj(\psi(T, c))=\maj(T)+c.
	 		\label{eq-psi-2}
	 	\end{gather}
	 				 
	 			\end{theorem}
	 			\begin{proof}
	 		In view of  Lemma \ref{lem-psi-1},   the map $\psi$ maps the pair $(T,c)$ with $T\in \B_{n-1}$ and $0\leq c\leq n-1$ to a tree $Q\in \B_n$ such that the resulting tree $Q$ verifies (\ref{eq-psi-1}) and (\ref{eq-psi-2}).  
	 			By cardinality reasons,  in order to show that $\psi$ is a bijection, it suffices to show that the map $\psi$ is injective. 
	 
	 			Let $\psi(T, c)=Q$, where $T\in \B_{n-1}$ and $c\in \{0,1,\dots, n-1\}$. To prove that the map $\psi$ is injective, it suffices to show that we can recover the pair $(T,c)$ from $Q$. Let $\beta(Q)=(\beta_1, \beta_2, \ldots, \beta_m)$. We consider two cases.
	 			
	 			\noindent\textbf{Case $\mathrm{(i')}$:} If the parent of node $n$ is a singleton internal node, define $T'$ to be the tree obtained from $Q$ by the following steps:
	 			\begin{itemize}
	 				\item Remove node $n$ from $Q$ together with its incident edge;
	 				\item modify the child‑type of the remaining singleton internal nodes so that $\beta(T')=(\beta_1, \beta_2, \ldots, \beta_{m-1})$.
	 			\end{itemize}
	 			
	 			\noindent\textbf{Case $\mathrm{(ii')}$:} Otherwise, define $T'$ to be the tree obtained from $Q$ by removing node $n$ together with its incident edge.
	 			
	 			According to the definition of $\psi$, one readily verifies that the construction in Case $\mathrm{(i')}$ reverses the procedures  described in Cases $\mathrm{(i)}$ and  $\mathrm{(ii)}$ for $\psi$, while the construction in Case $\mathrm{(ii')}$ reverses the procedures  described  in Cases $\mathrm{(iii)}$ and $\mathrm{(iv)}$ for $\psi$. Consequently, applying the above procedure yields the original tree $T$, so $T'=T$. By Lemma \ref{lem-psi-1}, we then obtain $c=\maj(Q)-\maj(T)$.   Hence, every image $Q$ uniquely determines the preimage $(T,c)$, which implies that $\psi$ is injective. This completes the proof.

	 				\end{proof}
	 				
	 			\subsection{Finishing the proof of Theorem \ref{thm-Phi}}
	 				This subsection is devoted to the construction of the map $\Phi$ as stated in Theorem \ref{thm-Phi}. 
	 				
	 				\noindent{\bf The construction of the map $\Phi$.}\\
	 				Given a permutation $\pi\in \mathfrak{S}_n$,  if $n=1$, then define $\Phi(\pi)$ to be the increasing binary tree with only one node with label $1$. Otherwise, let $\phi(\sigma, c)=\pi$, where $\sigma\in\mathfrak{S}_{n-1}$ and $0\leq c\leq n-1$. Then set
	 				$\Phi(\pi)=\psi(\Phi(\sigma), c)$.

	 				\noindent{\bf Proof of Theorem \ref{thm-Phi}.} Since both $\phi$ and $\psi$ are bijections, the map $\Phi$  induces a bijection between $\mathfrak{S}_n$ and $\B_n$. 
	 				Now we proceed to show that for any $\pi\in \mathfrak{S}_n$, we have
	 				$$
	 				(\des, \maj)\pi=(\l, \maj)\Phi(\pi). 
	 				$$
	 				Here we retain the notations in the definition of the map $\Phi$.  We proceed to prove the statement by induction on $n$. It is apparent that the statement holds for $n=1$. 
	 				Now we assume that $n\geq 2$ and that for any $\tau\in \mathfrak{S}_{n-1}$, we have  $$
	 				(\des, \maj)\tau=(\l, \maj)\Phi(\tau). 
	 				$$
	 				 By (\ref{eq-phi-2})  and (\ref{eq-psi-2}), we deduce that
	 				 $$
	 				 \maj(\Phi(\pi))=\maj(\Phi(\sigma))+c=\maj(\sigma)+c=\maj(\pi),
	 				 $$
	 				where the  second equality follows from the induction hypothesis.
	 				
	 				By  (\ref{eq-phi-1})  and (\ref{eq-psi-1}), we deduce that
	 				$$
	 					\l(\Phi(\pi))=\left\{ 
	 				\begin{array}{ll}
	 					\l(\Phi(\sigma)) & \text{if } 0\le c\le \l(\Phi(\sigma)),\\[2pt]
	 					\l(\Phi(\sigma))+1 & \text{otherwise},
	 				\end{array}
	 				\right. 
	 				$$
	 				and 
	 				
	 				$$
	 				\des(\pi)=\left\{ 
	 				\begin{array}{ll}
	 					\des(\sigma)& \text{if } 0\le c\le \des(\sigma),\\[2pt]
	 					\des(\sigma)+1 & \text{otherwise}.
	 				\end{array}
	 				\right. 
	 				$$
	 	This,  combined with the relation $\des(\sigma)=\l(\Phi(\sigma))$,  yields $\des(\pi)=\l(\Phi(\pi))$ as desired, completing the proof. \qed
	 				
	 				Now we are in the position to complete the proof of Theorem \ref{thm-gamma}.
	 				
	 					\noindent{\bf Proof of Theorem \ref{thm-gamma}.} 
	 			 
	 				In view of Theorem \ref{thm-Phi}, we deduce that
	 				$$
	 				A_n(t,q)=	\sum\limits_{\pi\in \mathfrak{S}_n}t^{\des(\pi)}q^{\maj(\pi)}=	\sum\limits_{T\in \B_n}t^{\l(T)}q^{\maj(T)}.
	 				$$
	 				This,  together with (\ref{eq-expansion}),  leads to the proof of Theorem \ref{thm-gamma}. 
	 				
	 				\qed
	 			
	 				 \section{A generalized Foata–Strehl action on  increasing binary trees of type $B$}
	 				 In this section, we will introduce a group action on increasing binary trees of type $B$, thereby proving Theorem  \ref{thm-gamma-B}.

	 				 Let $$ \widetilde{\B}_{n,k}=\{T\in \widetilde{\B}_{n}\mid  \leaf(T)=k+1\}.$$
	 				 By (\ref{eq-sint}),   for any $T\in \widetilde{\B}_{n,k}$,  we have $\sint(T)=n-2k$.
	 				 Let  $T\in \widetilde{\B}_{n,k}$ with $\Sint(T)=\{u_1, u_2, \ldots, u_{n-2k}\}$, where $|u_1|<|u_2|<\ldots< |u_{n-2k}|$.  Given $x\in[n-2k]$,    
	 			define  $\tau_x(T)$ to  be the tree obtained from $T$ by  flipping the sign of node $u_x$.
	 				
	 				\begin{lemma}\label{lem-alpha-B}
	 					Let $T\in \widetilde{\B}_{n,k}$. For any $x\in [n-2k]$, we have $\alpha_B(\tau_x(T))=\alpha_B(T)$. 
	 					\end{lemma}
	 				\begin{proof}
	 					Here we retain the notations in the definition of $\tau_x(T)$. 
	 					We prove the statement by induction on $n$. For $n=1$, both trees in
	 					$\widetilde{\mathcal T}_1$ have $\alpha_B=0$. Assume that $n\geq 2$
	 					and that the statement holds for all trees in
	 					$\widetilde{\mathcal T}_{n-1}$.

	 					Let $Q=\tau_x(T)$, and let $T'$ and $Q'$ be the trees obtained from
	 					$T$ and $Q$, respectively,  as constructed in the definition
	 					of $\alpha_B$.
	 					
	 					If $|u_x|=n$, then $u_x$ is the parent of node $\infty$, and
	 					$T'=Q'$. Since the corresponding increment is
	 					$2\operatorname{leaf}(T')-2$ in both trees, we have
	 					$\alpha_B(T)=\alpha_B(Q)$.
	 					If $u_x$ is the parent of the leaf $\pm n$,   the sign  of this 
	 					parent  is determined by the sign of $\pm n$ in both trees, so again $T'=Q'$.
	 			This implies that
	 					$\Leaf(T')=\Leaf(Q')$, and node $\pm n$
	 					has the same left/right position in $T$ and $Q$. Thus the first
	 					two cases in the definition of $\alpha_B$ give
	 					$\alpha_B(T)=\alpha_B(Q)$.
	 					
	 					In all other cases, $u_x$ remains a singleton internal node in
	 					$T'$.  Suppose that 
	 					$\Sint(T')=\{v_1, v_2, \ldots, v_m\}$  with $|v_1|<|v_2|<\cdots<|v_m|$ and that $v_y=u_x$.   Then
	 					$Q'=\tau_y(T')$, and hence the induction hypothesis gives
	 					\[
	 					\alpha_B(T')=\alpha_B(Q').
	 					\]
	 					Moreover, $T'$ and $Q'$ have the same leaves, and their singleton
	 					internal nodes have the same absolute values. If node $\pm n$ is a leaf,  then its parent is different
	 					from $u_x$, so the signs and left/right position used
	 					in the recursive increment are unchanged. If node $\pm n$
	 					is the parent of node $\infty$, the increment depends only
	 					on the number of leaves of $T'$ and $Q'$.   Consequently,
	 					\[
	 					\alpha_B(T)-\alpha_B(T')
	 					=
	 					\alpha_B(Q)-\alpha_B(Q').
	 					\]
	 					Together with $\alpha_B(T')=\alpha_B(Q')$, this yields
	 					$\alpha_B(T)=\alpha_B(Q)$, completing the proof.
	 				\end{proof}
	 				
	 				 From the definition of $\tau_x$,  one can easily check that the number of leaves  is invariant under the transformation  $\tau_x$.   This, combined with Lemma \ref{lem-alpha-B}, 
	 				 ensures that  the transformation  $\tau_x$  has the following desired properties. 
	 				 \begin{lemma}\label{lem-tau-prop}
	 				 		Let $T\in \widetilde{\B}_{n,k}$  with
	 				 		$\Sint(T)=\{u_1, u_2, \ldots, u_{n-2k}\}$  where  $|u_1|<|u_2|<\ldots< |u_{n-2k}|$. 
	 				 		 For any $x\in [n-2k]$,    we have $\tau_x(T)\in \widetilde{\B}_{n,k}$. Moreover,  the resulting tree $\tau_{x}(T)$ verifies the following properties.
	 				 	\begin{itemize}

	 				 		\item  If node $u_x$ is   negative, then we have 
	 				 		$$
	 				 		\fmaj(\tau_{x}(T))=\fmaj(T)-2\leaf(T)+3-2x,
	 				 		$$
	 				 		and 
	 				 		$$
	 				 		\omega(\tau_{x}(T))=\omega(T)-1.
	 				 		$$
	 				 		
	 				 		\item  If node $u_x$ is  positive, then we have 
	 				 		$$
	 				 		\fmaj(\tau_{x}(T))=\fmaj(T)+2\leaf(T)-3+2x,
	 				 		$$
	 				 		and 
	 				 		$$
	 				 		\omega(\tau_{x}(T))=\omega(T)+1.
	 				 		$$
	 				 	\end{itemize}
	 				 	
	 				 	\end{lemma}
	 				 
	 				 By the definition of $\tau_x$,  it is not hard to see that the transformation $\tau_x$ is an involution
	 				 on $\widetilde{\B}_{n,k}$ and $\tau_x$ and $\tau_y$ commute for all $x,y\in [n-2k]$.
	 				   For any  $S\subseteq [n-2k]$,  we can define the
	 				 function \[
	 				 \tau_S=\prod_{x\in S}\tau_x,
	 				 \]
	 				 where the product is the functional composition.      Hence the group $\mathbb{Z}_2^{n-2k}$
	 				 acts on $ \widetilde{\B}_{n,k}$ via the functions $\tau_S$ for $S\subseteq [n-2k]$.
	 				 Such an action is called the {\em generalized Foata–Strehl action  of type $B$} (GFS‑action of type $B$  for short) on increasing binary trees of type $B$.

	 				 \begin{lemma}\label{lem-b_{n,k}}
	 				 	For $n\geq 1$ and $k\geq 0$, we have
	 				 	$$
	 				 	b_{n,k}(q)=\sum_{T\in\widetilde{\mathcal T}_{n,k}^{*}}q^{\alpha_B(T)}.
	 				 	$$
	 				 	\end{lemma}

	 				 \begin{proof}
	 				 Let $$c_{n,k}(q)=\sum_{T\in\widetilde{\mathcal T}_{n,k}^{*}}q^{\alpha_B(T)}.$$  For $n\ge 1$, the only tree $T$ in
	 				 	$\widetilde{\mathcal T}_{n,0}^{*}$  consists of  the right chain
	 				 	$1, 2,\ldots, n,\infty$ and $\alpha_B(T)=0$. Hence $c_{n,0}(q)=1$. Moreover, $c_{n,k}(q)=0$ if $k<0$ or $k>\lfloor n/2\rfloor$.
	 				 	
	 				 	Let $1\le k\le\lfloor n/2\rfloor$ and $T\in\widetilde{\mathcal T}_{n,k}^{*}$.

	 				 	\noindent\textbf{Case (i).}  The node $\pm n$ is a child of a singleton internal node $u$.\\
	 				 	In this case, we can 
	 				 	construct an increasing binary tree  $T'$ of type $B$   by   removing  node $n$ (resp., $-n$) together with its incident edge  and  relabeling  $u$ as $|u|$ (resp., $-|u|$).
	 				 	Clearly, we have $T'\in \widetilde{\mathcal{T}}_{n-1,k}^{*}$. 
	 				 	
	 				 	 Conversely,   for each
	 				 	$T'\in\widetilde{\mathcal{T}}_{n-1,k}^{*}$,  we can recover a tree $T\in \widetilde{\mathcal{T}}_{n,k}^{*}$ as follows:
	 				 	\begin{itemize}
	 				 	\item Insert the node $n$ (resp. $-n$)   at the two positions of each finite leaf $\ell$ if $\ell$ is positive (resp. negative).
	 				 	\item Relabel node $\ell$ by $|\ell|$.
	 				 	\end{itemize}
	 				 	 By the first two cases in the definition of $\alpha_B(T)$, the corresponding increments are $0, 1, \ldots, 2k-1.$
	 				 	 Thus, each
	 				 	position gives exactly one tree in $\widetilde{\mathcal T}_{n,k}^{*}$, and the contribution of this case is$$ \sum_{r=0}^{2k-1}q^r c_{n-1,k}(q) = [2k]_q c_{n-1,k}(q).$$

	 				 		\noindent\textbf{Case (ii).}  The node $n$ is  the parent of node $\infty$.\\
	 				 	In this case, we can 
	 				 	construct an increasing binary tree $T'$ of type $B$  by   removing  node  $\infty$ and relabeling node $n$ by $\infty$. 
	 				 	Clearly, we have $T'\in \widetilde{\mathcal{T}}_{n-1,k}^{*}$. 
	 				 	
	 				 	Conversely,   for each
	 				 	$T'\in\widetilde{\mathcal{T}}_{n-1,k}^{*}$, the node $n$ can be inserted into the right arm of $T'$ so that the new inserted node becomes the parent of the node $\infty$.   By  the definition of $\alpha_B(T)$, the corresponding increment is given by $2k  $.  Thus, the contribution of this case is$$ q^{2k} c_{n-1,k}(q).$$

	 				 		\noindent\textbf{Case (iii).}  The node $\pm n$ is a child of an internal node $u$ which  has two children.\\
	 				 	In this case, we can 
	 				 	construct an increasing binary tree $T'$ of type $B$ by   removing  node  $\pm n$ and changing   the sign of  the node $u$ to be positive.   Then $T'\in\widetilde{\mathcal T}_{n-1,k-1}^{*}$.

	 				 	Conversely,  for any tree $T'\in\widetilde{\mathcal T}_{n-1,k-1}^{*}$, 
	 				 	let $\Sint(T')=\{u_1, u_2, \ldots, u_{M}\}$ with $M=n+1-2k$ and $|u_1|<|u_2|<\ldots<|u_M|$.  Then we can 
	 				 	 choose  $u_i$ and  change its sign to negative or positive, and then attach  node $-n$ or $n$ to its empty child position.  Lemma \ref{lem-alpha-B} states that changing the sign of this singleton internal node $u_i$ does not change $\alpha_B(T')$. The four increments of $\alpha_B$ are given by
	 				 	\[
	 				 	\begin{array}{ll}
	 				 		2(M-i)+2k-1, & 2(M-i)+2k,\\[2mm]
	 				 		4k+2i-4,     & 4k+2i-3.
	 				 	\end{array}
	 				 	\]
	 				 	Hence, the contribution is given by
	 				 	\[
	 				 	\begin{aligned}
	 				 		&\sum_{i=1}^{M} \left( q^{2(M-i)+2k-1}+q^{2(M-i)+2k}
	 				 		+q^{4k+2i-4}+q^{4k+2i-3}
	 				 		\right)c_{n-1,k-1}(q)\\
	 				 		&\qquad =q^{2k-1}(1+q)(1+q^{2k-1}) [M]_{q^2}c_{n-1,k-1}(q).
	 				 	\end{aligned}
	 				 	\]
	 				 	Hence,
	 				 	\[
	 				 	\begin{aligned}
	 				 		c_{n,k}(q) ={}&[2k+1]_q c_{n-1,k}(q)\\
	 				 		&+q^{2k-1}(1+q)(1+q^{2k-1}) [n+1-2k]_{q^2}c_{n-1,k-1}(q).
	 				 	\end{aligned}.
	 				 	\]
	 				 	Hence, $c_{n,k}$ and $b_{n,k}$ satisfy the same recurrence relation and initial conditions. This gives
	 				 	\[
	 				 	b_{n,k}(q)=
	 				 	\sum_{T\in\widetilde{\mathcal T}_{n,k}^{*}}q^{\alpha_B(T)},
	 				 	\]
	 				 	completing the proof.
	 				 	\end{proof}

	 				 	Now we are in the position to complete the proof of Theorem \ref{thm-gamma-B}.

	 				 	\noindent{\bf Proof of  Theorem \ref{thm-gamma-B}.} First we aim to verify the following expansion:
	 				 	\begin{equation}\label{eq-expansion-B}
	 				 		\sum\limits_{T\in \widetilde{\B}_n}t^{\omega(T)}q^{\fmaj(T)}=\sum\limits_{k=0}^{\lfloor {n/2}\rfloor}  
	 				 		b_{n, k}(q)t^{k}(-tq^{2k+1}; q^2)_{n-2k},
	 				 		\end{equation}
	 				 		where 
	 				 		$$
	 				 		b_{n,k}(q)=\sum_{T\in \widetilde{\B}^*_{n,k}}q^{\alpha_B(T)}.
	 				 		$$

	 				 		Recall that $\widetilde{\B}^*_{n,k}$ denotes the set of increasing binary trees of type $B$ in $\widetilde{\B}_n$ without any negative  singleton internal nodes and with exactly $k+1$ leaves.
	 				 		Given a tree  $T\in \widetilde{\B}^*_{n,k}$, we define  $\mathrm{Orbit}(T)$ to be the set of trees $S$ that can be transformed
	 				 		to $T$ via the GFS-action  of type $B$.  
	 				 	 In view of Lemma \ref{lem-tau-prop}, it follows that  the number of  leaves is invariant  under GFS-action  of type $B$.  Hence, the GFS-action of type $B$
	 				 	divides $\widetilde{\B}_{n,k}$ into disjoint orbits. 
	 				 	Moreover, each orbit contains a unique tree in $\widetilde{\B}^*_{n,k}$, obtained by replacing every negative singleton internal node with a positive one.
	 				 	By (\ref{eq-sint}),  we have
	 				 	\begin{equation}\label{eq-bar-T-1-B}
	 				 		\sint(T)=n+2-2\leaf(T)=n-2k
	 				 	\end{equation}
	 				 	for any $T\in \widetilde{\B}^*_{n,k}$.
	 				 	Given that  $\nsint(T)=0$ and $\leaf(T)=k+1$, we deduce that
	 				 	 	\begin{equation}\label{eq-bar-T-2-B}
	 				 	 	\omega(T)=\leaf(T)-1+\nsint(T)=k
	 				 	 \end{equation}
	 				 	for any $T\in \widetilde{\B}^*_{n,k}$.
	 				 	Then, by Lemma~\ref{lem-tau-prop}, we deduce that
	 				 		\[
	 				 	\begin{aligned}
	 				 		\sum\limits_{T\in \widetilde{\B}_n}t^{\omega(T)}q^{\fmaj(T)}
	 				 		&= \sum_{k\geq 0}\sum_{T\in\widetilde{\B}^*_{n,k} }\sum_{S\in \mathrm{Orbit}(T)} t^{\omega(S)}q^{\fmaj(S)}\\
	 				 		&= \sum_{k\geq 0}\sum_{T\in\widetilde{\B}^*_{n,k} }t^{k}q^{\fmaj(T)}\prod\limits_{i=1}^{n-2k}(1+tq^{2k+2i-1})\\
	 				 		&= \sum_{k\geq 0}\sum_{T\in\widetilde{\B}^*_{n,k} }t^{k}q^{\alpha_B(T)}\prod\limits_{i=1}^{n-2k}(1+tq^{2k+2i-1})\\
	 				 		&=  \sum_{k\geq 0} b_{n,k}(q)t^{k}(-tq^{2k+1}; q^2)_{n-2k},
	 				 	\end{aligned}
	 				 	\]
	 				 	where the second equality follows from (\ref{eq-bar-T-1-B}) and (\ref{eq-bar-T-2-B}), the third equality follows from    the relation $\fmaj(T)=\alpha_B(T)$, and 
	 			the last equality follows from Lemma \ref{lem-b_{n,k}}. 
	 				 	Then the first equality stated in (\ref{eq-gamma-2})  follows directly from (\ref{eq-expansion-B}) and  (\ref{eq-Han-B}), completing the proof. \qed
	 				 	
	 				 	\section{The divisibility of $\gamma$-coefficients}
	 				 	This section is devoted to the combinatorial interpretations  of  the quotients 
	 				 	$a_{n,k}(q)/ (-q;q)_{k-1}$ and $b_{n,k}(q)/(1+q)^k(-q; q^2)_k$ and the polynomial $G^*_{2n}(q)$. 
	 				 	
	 				 	\noindent{\bf Proof of Theorem \ref{thm-d_{n,k}(q)}.}
	 				 		For $n\geq1$, the unique tree $T$ in $\mathcal{AT}_{n,1}$ is the
	 				 		right chain whose nodes are labeled $1,2,\ldots,n$ from the root
	 				 		to the leaf. By the definition of $\alpha(T)$, we have $\alpha(T)=0$ and thus 
	 				 		\[
	 				 		d_{n,1}(q)=1.
	 				 		\]
	 				 		The initial condition and recurrence \eqref{eq-a_{n,k}(q)} also give
	 				 		$a_{n,1}(q)=1$.
	 				 		
	 				 		A binary tree with $k$ leaves has $k-1$ internal nodes with two
	 				 		children, and therefore has at least $2k-1$ nodes. Thus
	 				 		$\mathcal{AT}_{n,k}$ is empty whenever
	 				 		$k>\lfloor(n+1)/2\rfloor$. In this case, $d_{n,k}(q)=0$, while
	 				 		$a_{n,k}(q)=0$ by initial condition in
	 				 		\eqref{eq-a_{n,k}(q)}. It remains to consider
	 				 		\[
	 				 		2\leq k\leq\lfloor(n+1)/2\rfloor.
	 				 		\]
	 				 		
	 				 		Let $T\in\mathcal{AT}_{n,k}$, and let $\widehat{T}$ be the tree
	 				 		obtained from $T$ by removing node $n$ together with its incident
	 				 		edge. Since $T$ is increasing, node $n$ must be a leaf. Let $u$
	 				 		be its parent. We proceed by considering two cases.
	 				 		
	 				 		\noindent{\bf Case 1.} The node $u$ is a singleton internal node.
	 				 		
	 				 		By the definition of an Andr\'e tree, node $n$ is the right child
	 				 		of $u$. Removing node $n$ turns $u$ into a leaf, so
	 				 		\[
	 				 		\widehat{T}\in\mathcal{AT}_{n-1,k}.
	 				 		\]
	 				 		Conversely, let $\widehat{T}\in\mathcal{AT}_{n-1,k}$, and write
	 				 		\[
	 				 		\Leaf(\widehat{T})
	 				 		=\{\ell_1<\ell_2<\cdots<\ell_k\}.
	 				 		\]
	 				 		For each $1\leq i\leq k$, attaching node $n$ as the right child
	 				 		of $\ell_i$ gives a tree in $\mathcal{AT}_{n,k}$ belonging to
	 				 		Case~1. Clearly, the deletion and insertion procedures are inverse
	 				 		to each other. By the definition of $\alpha(T)$, the corresponding
	 				 		increment is $i-1$. Hence the contribution of this case is
	 				 		\[
	 				 		\sum_{i=1}^{k}q^{i-1}d_{n-1,k}(q)
	 				 		=[k]_q\,d_{n-1,k}(q).
	 				 		\]
	 				 		
	 				 		\noindent{\bf Case 2.} The node $u$ has two children.
	 				 		
	 				 		The other child of $u$ has a label smaller than $n$. Since the
	 				 		smaller child of every internal node in an Andr\'e tree is its
	 				 		right child, node $n$ must be the left child of $u$. Removing
	 				 		node $n$ turns $u$ into a singleton internal node with a right
	 				 		child. Therefore,
	 				 		\[
	 				 		\widehat{T}\in\mathcal{AT}_{n-1,k-1}.
	 				 		\]
	 				 		
	 				 		Conversely, let $\widehat{T}\in\mathcal{AT}_{n-1,k-1}$, and write
	 				 		\[
	 				 		\Sint(\widehat{T})
	 				 		=\{u_1<u_2<\cdots<u_m\}.
	 				 		\]
	 				 		Every singleton internal node of an Andr\'e tree has only a right
	 				 		child. For each $1\leq i\leq m$, attach node $n$ as the left child
	 				 		of $u_i$. Since the right child of $u_i$ has a label smaller than
	 				 		$n$, the resulting tree belongs to $\mathcal{AT}_{n,k}$. Again,
	 				 		the deletion and insertion procedures are inverse to each other.
	 				 		
	 				 		Note that the tree $\widehat{T}$ has $k-1$ leaves.  Then, by (\ref{eq-sint}), we have 
	 				 		\[
	 				 		m=n+2-2k.
	 				 		\]
	 				 		By the definition of $\alpha(T)$, attaching node $n$ as the left
	 				 		child of $u_i$ increases $\alpha$ by
	 				 		\[
	 				 		\leaf(\widehat{T})+m-i=k-1+m-i.
	 				 		\]
	 				 		Thus the contribution of this case is
	 				 		\[
	 				 		\begin{aligned}
	 				 			\sum_{i=1}^{m}q^{k-1+m-i}d_{n-1,k-1}(q)
	 				 			&=q^{k-1}[m]_q\,d_{n-1,k-1}(q)\\
	 				 			&=q^{k-1}[n+2-2k]_q\,d_{n-1,k-1}(q).
	 				 		\end{aligned}
	 				 		\]
	 				 		
	 				 		Combining the two cases, we obtain
	 				 		\[
	 				 		d_{n,k}(q)
	 				 		=[k]_q\,d_{n-1,k}(q)
	 				 		+q^{k-1}[n+2-2k]_q\,d_{n-1,k-1}(q).
	 				 		\]
	 				 		Using
	 				 		\[
	 				 		(-q;q)_{k-1}
	 				 		=(1+q^{k-1})(-q;q)_{k-2},
	 				 		\]
	 				 		we deduce that
	 				 		\[
	 				 		\begin{aligned}
	 				 			(-q;q)_{k-1}d_{n,k}(q)
	 				 			={}&[k]_q\,(-q;q)_{k-1}d_{n-1,k}(q)\\
	 				 			&+(1+q^{k-1})q^{k-1}[n+2-2k]_q\,
	 				 			(-q;q)_{k-2}d_{n-1,k-1}(q).
	 				 		\end{aligned}
	 				 		\]
	 				 		This is precisely the recurrence \eqref{eq-a_{n,k}(q)} for
	 				 		$a_{n,k}(q)$. Since the initial   conditions also
	 				 		agree,  we deduce that
	 				 		\[
	 				 		a_{n,k}(q)=(-q;q)_{k-1}d_{n,k}(q),
	 				 		\]
	 				 		completing the proof. \qed

	 				  \noindent{\bf Proof of Theorem  \ref{thm-e_{n,k}(q)}.}
	 				 	Recall that $b_{n,k}(q)$ satisfies the recurrence
	 				 	\[
	 				 	\begin{aligned}
	 				 		b_{n,k}(q)={}&[2k+1]_q b_{n-1,k}(q)\\
	 				 		&+(1+q)q^{2k-1}(1+q^{2k-1})
	 				 		[n+1-2k]_{q^2}b_{n-1,k-1}(q),
	 				 	\end{aligned}
	 				 	\]
	 				 	with $b_{1,0}(q)=1$, and $b_{n,k}(q)=0$ whenever
	 				 	$k<0$ or $k>\lfloor n/2\rfloor$.
	 				 	
	 				 	For $n\geq 1$, the only tree $T$ in $\NT_{n,0}$ is the right chain
	 				 	whose nodes are labeled
	 				 	\[
	 				 	1,2,\ldots,n,\infty
	 				 	\]
	 				 	from the root to the leaf. By the definition of $\alpha_B(T)$,  we have $\alpha_B(T)=0$. The initial condition and recurrence also give
	 				 	$b_{n,0}(q)=1$. Thus
	 				 	\[
	 				 	e_{n,0}(q)=1=b_{n,0}(q).
	 				 	\]
	 				 	
	 				 	A binary tree with $k+1$ leaves has $k$ internal nodes with two
	 				 	children, and hence has at least $2k+1$ nodes. Since every tree in
	 				 	$\widetilde{\mathcal T}_n$ has $n+1$ nodes, $\mathsf{NT}_{n,k}$ is
	 				 	empty whenever $k>\lfloor n/2\rfloor$. It remains to consider
	 				 	\[
	 				 	1\leq k\leq\lfloor n/2\rfloor.
	 				 	\]
	 				 	
	 				 	Let $T\in\mathsf{NT}_{n,k}$. Since every finite node of $T$ has a
	 				 	positive label, the node with absolute value $n$ is labeled $n$.
	 				 	Moreover, the absolute values of the labels increase along every path
	 				 	from the root to a leaf. Thus node $n$ has no finite child, so it is
	 				 	either a leaf or the parent of node $\infty$. Let $T'$ be the tree
	 				 	obtained from $T$ in the recursive definition of $\alpha_B(T)$. We
	 				 	consider the following three cases.
	 				 	
	 				 	\medskip
	 				 	\noindent{\bf Case 1.} The node $n$ is a leaf whose parent is a
	 				 	singleton internal node.
	 				 	
	 				 	Removing node $n$ turns its parent into a positive leaf, and hence
	 				 	$T'\in\mathsf{NT}_{n-1,k}$. Conversely, let
	 				 	$T'\in\mathsf{NT}_{n-1,k}$ and write
	 				 	\[
	 				 	\Leaf(T')=\{\ell_1<\ell_2<\cdots<\ell_k<\infty\}.
	 				 	\]
	 				 	For each $1\leq i\leq k$, attaching node $n$ as the left or right
	 				 	child of $\ell_i$ gives a tree in $\mathsf{NT}_{n,k}$ in which the node $n$ is a leaf whose parent is a
	 				 	singleton internal node.  The deletion and insertion procedures are inverses to each
	 				 	other. By the first two cases in the definition of $\alpha_B(T)$, the
	 				 	corresponding increments are $2i-2$ and $2i-1$. Therefore, the
	 				 	contribution of this case is given by
	 				 	\[
	 				 	\sum_{i=1}^{k}\bigl(q^{2i-2}+q^{2i-1}\bigr)e_{n-1,k}(q)
	 				 	=[2k]_q e_{n-1,k}(q).
	 				 	\]
	 				 	
	 				 	\medskip
	 				 	\noindent{\bf Case 2.} The node $n$ is the parent of node $\infty$.
	 				 	
	 				 	Removing node $\infty$ and relabeling node $n$ by $\infty$ gives a
	 				 	tree $T'\in\mathsf{NT}_{n-1,k}$. Conversely, every tree in
	 				 	$\mathsf{NT}_{n-1,k}$ gives a unique tree in this case by inserting
	 				 	node $n$ immediately above node $\infty$ on the right arm. Since
	 				 	$\leaf(T')=k+1$, the corresponding increment of $\alpha_B$ is
	 				 	\[
	 				 	2\leaf(T')-2=2k.
	 				 	\]
	 				 	Thus the contribution of this case  is
	 				 	\[
	 				 	q^{2k}e_{n-1,k}(q).
	 				 	\]
	 				 	
	 				 	\medskip
	 				 	\noindent{\bf Case 3.} The node $n$ is a leaf whose parent has two
	 				 	children.
	 				 	
	 				 	Removing node $n$ turns its parent into a positive singleton internal
	 				 	node. Hence $T'\in\mathsf{NT}_{n-1,k-1}$. Conversely, let
	 				 	$T'\in\mathsf{NT}_{n-1,k-1}$ and write
	 				 	\[
	 				 	\Sint(T')=\{u_1,u_2,\ldots,u_M\},
	 				 	\qquad u_1<u_2<\cdots<u_M.
	 				 	\]
	 				 Note that 	the tree $T'$ has $n$ nodes and $k$ leaves.  By (\ref{eq-sint}), it follows that
	 				 $
	 				   M=n+1-2k.
	 				 	$
	 				 	For each $1\leq i\leq M$, attaching node $n$ at the empty child
	 				 	position of $u_i$ gives a tree in $\mathsf{NT}_{n,k}$ belonging to
	 				 	this case. Again, the deletion and insertion procedures are inverses
	 				 	 of each other. Since $u_i$ is positive and $\leaf(T')=k$, the
	 				 	corresponding increment of $\alpha_B$ is given by 
	 				 	\[
	 				 	2\leaf(T')-1+2(M-i)=2k-1+2(M-i).
	 				 	\]
	 				 	Therefore, the contribution of  this case  is
	 				 	\[
	 				 	\begin{aligned}
	 				 		\sum_{i=1}^{M}q^{2k-1+2(M-i)}e_{n-1,k-1}(q)
	 				 		&=q^{2k-1}[M]_{q^2}e_{n-1,k-1}(q)\\
	 				 		&=q^{2k-1}[n+1-2k]_{q^2}e_{n-1,k-1}(q).
	 				 	\end{aligned}
	 				 	\]
	 				 	
	 				 	  Combining
	 				 the above three  cases gives
	 				 	\[
	 				 	e_{n,k}(q)=[2k+1]_q e_{n-1,k}(q)
	 				 	+q^{2k-1}[n+1-2k]_{q^2}e_{n-1,k-1}(q).
	 				 	\]
	 				 	Since
	 				 	\[
	 				 	(1+q)^k(-q;q^2)_k
	 				 	=(1+q)(1+q^{2k-1})(1+q)^{k-1}(-q;q^2)_{k-1},
	 				 	\]
	 				 	we obtain
	 				 	\[
	 				 	\begin{aligned}
	 				 		&(1+q)^k(-q;q^2)_k e_{n,k}(q)\\
	 				 		&\quad=[2k+1]_q(1+q)^k(-q;q^2)_k e_{n-1,k}(q)\\
	 				 		&\qquad +(1+q)q^{2k-1}(1+q^{2k-1})[n+1-2k]_{q^2}
	 				 		(1+q)^{k-1}(-q;q^2)_{k-1}e_{n-1,k-1}(q).
	 				 	\end{aligned}
	 				 	\]
	 				 	Thus $(1+q)^k(-q;q^2)_k e_{n,k}(q)$ satisfies the same recurrence relation and the same initial conditions as
	 				 	$b_{n,k}(q)$.  Therefore, 
	 				 	induction on $n$ gives
	 				 	\[
	 				 	b_{n,k}(q)=(1+q)^k(-q;q^2)_k e_{n,k}(q)
	 				 	\]
	 				  as desired. \qed

	 			\noindent{\bf Proof of Theorem \ref{thm:positive-G-star}.}	 
	 				Since $G_0^{*}(q)=1$, the assertion is immediate for $n=0$. Let
	 				$n\geq 1$. Recall that
	 				\[
	 				B_m(t,q)=\sum_{k=0}^{\lfloor m/2\rfloor}
	 				b_{m,k}(q)t^k(-tq^{2k+1};q^2)_{m-2k}.
	 				\]
	 				Taking $m=2n$ and $t=-q^{-2n-1}$ gives
	 				\[
	 				\begin{aligned}
	 					B_{2n}(-q^{-2n-1},q)
	 					=\sum_{k=0}^{n}{}&b_{2n,k}(q)(-q^{-2n-1})^k\\
	 					&\quad\times(q^{2k-2n};q^2)_{2n-2k}.
	 				\end{aligned}
	 				\]
	 				For $0\leq k<n$, the factor
	 			\[
	 			(q^{2k-2n};q^2)_{2n-2k}
	 			=\prod_{j=0}^{2n-2k-1}
	 			(1-q^{2k-2n+2j})
	 			\]
	 			vanishes, since the factor corresponding to $j=n-k$ is
	 			$1-q^0$. Thus only the term $k=n$ remains, and hence
	 			\[
	 			B_{2n}(-q^{-2n-1},q)
	 			=(-1)^nq^{-n(2n+1)}b_{2n,n}(q).
	 			\]
	 			It follows that
	 			\[
	 			E_{2n}^{*}(q)=q^{-n^2}b_{2n,n}(q).
	 			\]
	 				 
	 				Then Theorem~\ref{thm-e_{n,k}(q)} yields
	 				\[
	 				\begin{aligned}
	 					G_{2n}^{*}(q)
	 					&=\frac{q^{-n^2}b_{2n,n}(q)}{(1+q)^n(-q;q^2)_n}\\
	 					&=q^{-n^2}e_{2n,n}(q)\\
	 					&=\sum_{T\in\mathsf{NT}_{2n,n}}q^{\alpha_B(T)-n^2}.
	 				\end{aligned}
	 				\]
	 				
	 			It remains to verify that the last expression is a polynomial and that
	 			none of its coefficients vanishes. For $m\geq 1$ and
	 			$0\leq k\leq\lfloor m/2\rfloor$, define
	 			\[
	 			h_{m,k}(q)=q^{-k^2}e_{m,k}(q),
	 			\]
	 			and set $h_{m,k}(q)=0$ whenever $k<0$ or $k>\lfloor m/2\rfloor$.
	 			In the proof of Theorem~\ref{thm-e_{n,k}(q)}, we have proved that 
	 			\[
	 			e_{m,k}(q)=[2k+1]_q e_{m-1,k}(q)
	 			+q^{2k-1}[m+1-2k]_{q^2}e_{m-1,k-1}(q).
	 			\]
	 			Since
	 			\[
	 			k^2-(k-1)^2=2k-1,
	 			\]
	 			we obtain
	 			\begin{equation}\label{eq:h-recurrence}
	 				h_{m,k}(q)=[2k+1]_q h_{m-1,k}(q)
	 				+[m+1-2k]_{q^2}h_{m-1,k-1}(q).
	 			\end{equation}
	 			Moreover, $h_{m,0}(q)=1$ for every $m\geq 1$.
	 			
	 			We claim that, for $0\leq k\leq\lfloor m/2\rfloor$,
	 			$h_{m,k}(q)$ is a polynomial of degree
	 			\[
	 			D_{m,k}=k(2m-3k-1)
	 			\]
	 			and has positive coefficients in every degree from $0$ to $D_{m,k}$.
	 			We proceed by induction on $m$. The case $m=1$ is clear. Let $m\geq2$
	 			and assume that the assertion holds for $m-1$. The assertion is also
	 			clear when $k=0$. Suppose first that $m=2k$. Then
	 			$h_{m-1,k}(q)=0$, and
	 			\eqref{eq:h-recurrence} reduces to
	 			\[
	 			h_{2k,k}(q)=h_{2k-1,k-1}(q).
	 			\]
	 			The result follows from the induction hypothesis, since
	 			\[
	 			D_{2k-1,k-1}=k(k-1)=D_{2k,k}.
	 			\]
	 			
	 			Now suppose that $m>2k$. By the induction hypothesis,
	 			$h_{m-1,k}(q)$ has positive coefficients in every degree from $0$ to
	 			$D_{m-1,k}$. Since
	 			\[
	 			[2k+1]_q=1+q+\cdots+q^{2k}
	 			\]
	 			and
	 			\[
	 			D_{m-1,k}+2k=D_{m,k},
	 			\]
	 			the first term on the right-hand side of \eqref{eq:h-recurrence} is a
	 			polynomial of degree $D_{m,k}$ and has positive coefficients in every
	 			degree from $0$ to $D_{m,k}$. By the induction hypothesis, the second
	 			term has nonnegative coefficients and degree $D_{m,k}$, since
	 			\[
	 			D_{m-1,k-1}+2(m-2k)=D_{m,k}.
	 			\]
	 			Thus $h_{m,k}(q)$ is a polynomial of degree $D_{m,k}$ and has positive
	 			coefficients in every degree from $0$ to $D_{m,k}$, proving the claim.
	 			
	 			Finally,
	 			\[
	 			G_{2n}^{*}(q)=h_{2n,n}(q)
	 			\qquad\text{and}\qquad
	 			D_{2n,n}=n(n-1).
	 			\]
	 			The result follows.
	 				 
   			    \section*{Acknowledgments}
   This work was supported by the National Natural Science
   Foundation of China grant 12471318.


\end{document}